\documentclass[letterpaper,12pt]{amsart}
\usepackage{amsmath,amssymb,amsxtra,amsthm, amstext, amscd,amsfonts,fancyhdr, hyperref,enumerate}
\hypersetup{
colorlinks=true,
urlcolor=black,
citecolor=blue,
linkcolor=blue,
}

\usepackage[OT2,T1]{fontenc}
\DeclareSymbolFont{cyrletters}{OT2}{wncyr}{m}{n}
\DeclareMathSymbol{\Sha}{\mathalpha}{cyrletters}{"58}

\newcommand{\bC}{{\mathbb{C}}}

\newcommand{\bN}{{\mathbb{N}}}

\newcommand{\bQ}{{\mathbb{Q}}}
\newcommand{\bR}{{\mathbb{R}}}

\newcommand{\bZ}{{\mathbb{Z}}}

  \newcommand{\B}{{\mathcal{B}}}
  
  \newcommand{\D}{{\mathcal{D}}}

\renewcommand{\L}{{\mathcal{L}}}

\renewcommand{\O}{{\mathcal{O}}}

  \newcommand{\Q}{{\mathcal{Q}}}
  \newcommand{\R}{{\mathcal{R}}}
\renewcommand{\S}{{\mathcal{S}}}
  \newcommand{\T}{{\mathcal{T}}}

\newcommand{\AND}{\text{ and }}

\newcommand{\OR}{\text{ or }}

\newcommand{\fC}{\mathfrak{C}}

\newcommand{\Gal}{\operatorname{Gal}}

\newcommand{\GL}{\operatorname{GL}}

\newcommand{\fF}{\mathfrak{F}}

\newcommand{\sm}{\mathrm{\tiny sm}}

\newcommand{\ol}{\overline}

\newcommand{\upchi}{{\raise.35ex\hbox{$\chi$}}}

\newcommand{\Vol}{\operatorname{Vol}}

\newcommand{\BS}{\mathrm{\tiny BS}}

\makeatletter
\@namedef{subjclassname@2010}{%
  \textup{2010} Mathematics Subject Classification}
\makeatother

\newtheorem{theorem}{Theorem}[section]
\newtheorem{corollary}[theorem]{Corollary}
\newtheorem{proposition}[theorem]{Proposition}
\newtheorem{lemma}[theorem]{Lemma}

\theoremstyle{definition}
\newtheorem{definition}[theorem]{Definition}

\theoremstyle{remark}

\numberwithin{equation}{section}

\allowdisplaybreaks
\begin{document}

\title[Binary quartic forms with small Galois groups]{Binary quartic forms with bounded invariants\\and small Galois groups}

\author{Cindy (Sin Yi) Tsang}
\address{Yau Mathematical Sciences Center\\
Tsinghua University\\
Beijing, P. R. China}
\email{sinyitsang@math.tsinghua.edu.cn}

\author{Stanley Yao Xiao}
\address{Mathematical Institute \\
University of Oxford \\
Andrew Wiles Building \\
Radcliffe Observatory Quarter \\
Woodstock Road \\
Oxford\\  OX2 6GG }
\email{stanley.xiao@maths.ox.ac.uk}
\indent

\date{\today}


\begin{abstract}In this paper, we consider integral and irreducible binary quartic forms whose Galois group is isomorphic to a subgroup of the dihedral group of order eight. We first show that the set of all such forms is a union of families indexed by integral binary quadratic forms $f(x,y)$ of non-zero discriminant. Then, we shall enumerate the $\GL_2(\bZ)$-equivalence classes of all such forms associated to a fixed $f(x,y)$.
\end{abstract}

\maketitle


\vspace{-5mm}

\tableofcontents

\section{Introduction}
\label{Intro}

The problem of enumerating $\GL_2(\bZ)$-equivalence classes of integral and irreducible binary forms of a fixed degree has a long history. The quadratic and cubic cases were solved in \cite{Gau, Sieg} and \cite{Dav1, Dav2}, respectively, where the forms are ordered by the natural height, namely the discriminant $\Delta(-)$. The quartic case turns out to be much more challenging because while the ring of polynomial invariants for both binary quadratic and cubic forms is generated by $\Delta(-)$ as an algebra, that for binary quartic forms is generated by two independent invariants, usually denoted by $I(-)$ and $J(-)$. For
\begin{equation}\label{F generic}F(x,y) = a_4 x^4 + a_3 x^3y + a_2 x^2y^2 + a_1 xy^3 + a_0 y^4,\end{equation}
they are given by the explicit formulae
\begin{align*}
I(F)& = 12a_4 a_0 - 3a_3 a_1 + a_2^2,\\
J(F)& = 72a_4 a_2 a_0  + 9a_3 a_2 a_1 - 27a_4 a_1^2 - 27a_3^2 a_0 - 2a_2^3,
\end{align*}
which are of degrees two and three, respectively. In \cite{BhaSha}, instead of using the discriminant, Bhargava and Shankar introduced the height function
\begin{equation}\label{BS height} H_\BS(F) = \max\{|I(F)|^3, J(F)^2/4\}.\end{equation}
For $X>0$, let us define
\begin{align*} N_{\bZ}(X) &= \#\{[F]:\mbox{integral and irreducible binary}\\ &\hspace{3cm}\mbox{quartic forms $F$ such that  $H_\BS(F)\leq X$}\},\end{align*}
where $[-]$ denotes $\GL_2(\bZ)$-equivalence class. In \cite{BhaSha}, they proved that
\begin{equation}\label{BS formula}N_\bZ(X) = \frac{44 \zeta(2)}{135} X^{5/6} + O_\epsilon \left(X^{3/4 + \epsilon} \right)\mbox{ for any $\epsilon>0$}.\end{equation}
This is the first result ever obtained, and as far as we know, the only known result in the literature, for the quartic case. 

\subsection{Set-up and notation}\label{notation sec}

In this paper, we shall also be interested in the quartic case, but only the integral and irreducible binary quartic forms $F$ with \emph{small} Galois group $\Gal(F)$, which is defined to be the Galois group of the splitting field of $F(x,1)$ over $\bQ$. We know that $\Gal(F)$ is isomorphic to one of the following:
\begin{align*}
S_4 & = \mbox{the symmetric group on four letters},\\
A_4 & = \mbox{the alternating group on four letters},\\
D_4 & = \mbox{the dihedral group of order eight},\\
C_4 & = \mbox{the cyclic group of order four},\\
V_4 & = \mbox{the Klein-four group}.
\end{align*}
We shall say that $\Gal(F)$ is \emph{small} if it is isomorphic to $D_4,C_4$, or $V_4$. Recall that the \emph{cubic resolvent of $F$} is defined by
\[ \Q_F(x) = x^3 - 3I(F)x + J(F).\]
Then, equivalently, we have the classical characterization that for irreducible $F$
\[\Gal(F)\mbox{ is small if and only if $\Q_F(x)$ is reducible}.\]
It turns out that whether $\Gal(F)$ is small or not may also be characterized in terms of binary quadratic forms and the following so-called \emph{twisted action} of $\GL_2(\bR)$.\\

Given a complex binary form $\xi(x,y)$, let $\GL_2(\bR)$ act on it via
\[ \xi_T(x,y) = \frac{1}{\det(T)^{\deg\xi/2}}\xi(t_1x+t_2y, t_3x+ t_4y) \mbox{ for }T= \begin{pmatrix} t_1 & t_2 \\ t_3 & t_4\end{pmatrix}.\]
Observe that this is only an action up to sign when $\deg \xi$ is odd, in the sense that for $T_1,T_2\in\GL_2(\bR)$, we only have $\xi_{T_1T_2} = \pm (\xi_{T_1})_{T_2}$ in general. Now, given a real binary quadratic form $f(x,y) = \alpha x^2 + \beta xy + \gamma y^2$
with $\Delta(f)\neq0$, write
\[ M_f = \begin{pmatrix} \beta & 2\gamma \\ -2\alpha & -\beta \end{pmatrix}\]
for its associated matrix in $\GL_2(\bR)$. Its action on binary quartic forms clearly remain unchanged if we scale $f(x,y)$ by a constant in $\bR^\times$. In \cite{X}, the second-named author proved that for any real binary quartic form $F$ with $\Delta(F)\neq0$, elements of
\[ \{T\in\GL_2(\bR) : T\mbox{ is not a scalar multiple of }I_{2\times 2}\AND F_T = F\}\]
all arise from binary quadratic forms in this way; see Proposition~\ref{auto theorem}. Recall that an integral binary quadratic form is called \emph{primitive} if its coefficients are coprime. Using this result from \cite{X}, in Section~\ref{Galois gp sec}, we shall first show that:

\begin{theorem}\label{small char thm}Let $F$ be an integral binary quartic form with $\Delta(F)\neq0$. Then, the following are equivalent.
\begin{enumerate}[(1)]
\item $\Q_F(x)$ is reducible.
\item $F_T = F$ for some $T\in\GL_2(\bQ)$ which is not a scalar multiple of $I_{2\times 2}$.
\item $F_{M_f} = F$ for an integral and primitive binary quadratic form $f$ with $\Delta(f)\neq0$.
\end{enumerate}
Moreover, in the case that  $\Q_F(x)$ is reducible:
\begin{enumerate}[(a)]
\item If $\Delta(F)\neq\square$, then there is a unique such $f$ up to sign.
\item If $\Delta(F)=\square$, then there are exactly three such $f$ up to sign, among which one is definite and two are indefinite.
\end{enumerate}
\end{theorem}

Given a real binary quadratic form $f(x,y)$ with $\Delta(f)\neq0$, let us further make the following definitions. First put
\begin{align*} \label{VRf} 
V_{\bR,f} &= \{\mbox{real binary quartic forms $F$ such that $F_{M_f} = F$}\},\\
V_{\bZ,f}& = \{\mbox{integral binary quartic forms $F$ such that $F_{M_f} = F$}\}. \notag
\end{align*}
Clearly $V_{\bR,f}$ is a vector space over $\bR$ and $V_{\bZ,f}$ a lattice over $\bZ$. A straightforward calculation shows that $\dim_{\bR} V_{\bR,f}$ is three; see (\ref{abc family}) and (\ref{abc family 2}) below. Also, put
\[ V_{\bR,f}^0 = \{F\in V_{\bR,f}:\Delta(F)\neq0\}\AND V_{\bZ,f}^0 = \{F\in V_{\bZ,f}:\Delta(F)\neq0\}.\]
For $F\in V_{\bR,f}^0$, we shall define two new invariants as follows. As we shall see in (\ref{in V char}), there is a unique root $\omega_f(F)$ of $\Q_F(x)$ corresponding to $f$. Let $\omega'_f(F),\omega''_f(F)$ denote the other two roots of $\Q_F(x)$ and define
\begin{equation}\label{LK def}L_f(F) = \omega_f(F) \AND K_f(F) = -\omega'_f(F)\omega''_f(F).\end{equation}
By Proposition~\ref{explicit LK} below, they have degrees one and two, respectively, in the coefficients of $F$. Following (\ref{BS height}), let us define the \emph{height of $F$ associated to $f$} by
\[H_f(F) = \max\{L_f(F)^2, |K_f(F)|\}.\]
This is comparable to the height (\ref{BS height}) because by comparing coefficients in 
\[ x^3 - I(F)x + J(F) = (x-\omega_f(F))(x-\omega'_f(F))(x-\omega''_f(F)),\]
we easily deduce the relations
\begin{equation}\label{IJ family}
3I(F) = L_f(F)^2 + K_f(F) \AND J(F) = L_f(F) K_f(F),
\end{equation}
which in turn imply that 
\begin{equation}\label{H compare}  (H_f(F)/10)^3 \leq H_{\text{\tiny BS}}(F) \leq   H_f(F)^3.\end{equation} 
Let us note that
\begin{equation}\label{Delta LK}
\Delta(F) =  \frac{4I(F)^3 - J(F)^2}{27} = \left(\frac{L_f(F)^2 + 4K_f(F)}{9}\right)\left(\frac{2L_f(F)^2 - K_f(F)}{9}\right)^2,
\end{equation}
where the first equality is well-known, and the second equality holds by (\ref{IJ family}). Also, our height $H_f(-)$ is an invariant in the sense that for any $T\in\GL_2(\bR)$, we have
\[H_{f_T}(F_T) = H_f(F),\]
as shown in Proposition~\ref{LK invariant} below. This implies that the map
\begin{equation}\label{V bijection} V_{\bR,f} \longrightarrow V_{\bR,f_T}; \hspace{1em}F\mapsto F_T,\end{equation}
which is a well-defined bijection because $M_{f_T} = T^{-1}M_fT$, is height-preserving when restricted to the forms of non-zero discriminant.\\

Now, let us return to the integral and irreducible binary quartic forms with small Galois group. Write $V_{\bZ}^{\sm}$ for the set of all such forms and set
\[ V_{\bZ}^{\sm,\dagger} = \{F\in V_\bZ^{\sm}:\Gal(F)\not\simeq V_4\}.\]
By Theorem~\ref{small char thm}, we know that
\begin{align}\label{union}V_{\bZ}^{\sm} & = \bigcup_{f\in\fF^*} \{F\in V_{\bZ,f}^0:F\mbox{ is irreducible}\}, \\\notag
V_{\bZ}^{\sm,\dagger} & = \bigsqcup_{f\in\fF^*} \{F\in V_{\bZ,f}^0:F\mbox{ is irreducible and }\Gal(F)\not\simeq V_4\}, \end{align}
where $\fF^*$ denotes the set of all integral and primitive binary quadratic forms of non-zero discriminant, up to sign. In particular, given $F\in V_{\bZ}^{\sm,\dagger}$, there is a unique $f\in\fF^*$ such that $F\in V_{\bZ,f}^0$, and we may define the \emph{height of $F$} by setting
\[H(F) = H_f(F).\]
For $X>0$, let us define
\begin{align*}N_\bZ^\dagger(X) &= \#\{[F] : F\in V_\bZ^{\sm,\dagger}\mbox{ such that }H(F)\leq X\},\\
N_{\bZ,f}^\dagger(X) &= \#\{[F] : F\in V_{\bZ}^{\sm,\dagger}\cap V_{\bZ,f}^0\mbox{ such that } H(F)\leq X\}.\end{align*}
Then, by (\ref{V bijection}) and (\ref{union}), we have
\[N_\bZ^\dagger(X) = \sum_{f\in\fF} N_{\bZ,f}^\dagger(X),\]
where $\fF$ denotes a set of representatives of the $\GL_2(\bZ)$-equivalence classes on $\fF^*$. In Theorem~\ref{Small Gal MT}, which is our main result, for $f\in\fF^*$, we shall determine the asymptotic formula for $N_{\bZ,f}^\dagger(X)$. In fact, we shall consider the finer counts
\begin{align*}
N_{\bZ,f}^{(D_4)}(X) = \# \{[F] : F \in V_{\bZ}^{\sm} \cap V_{\bZ,f}^0 \mbox{ such that } \Gal(F) \simeq D_4 \AND H(F) \leq X\},\\
N_{\bZ,f}^{(C_4)}(X) = \# \{[F] : F \in V_{\bZ}^{\sm} \cap V_{\bZ,f}^0 \mbox{ such that } \Gal(F) \simeq C_4 \AND H(F) \leq X\},\\
N_{\bZ,f}^{(V_4)}(X) = \# \{[F] : F \in V_{\bZ}^{\sm} \cap V_{\bZ,f}^0 \mbox{ such that } \Gal(F) \simeq V_4 \AND H_f(F) \leq X\},\end{align*}
and show that the latter two are negligible compared to $N_{\bZ,f}^{(D_4)}(X)$. This means that most of the forms in $V_{\bZ}^{\sm}\cap V_{\bZ,f}^0$ have Galois group isomorphic to $D_4$. However, all of our error estimates depend upon $f$. Currently, we do not know how to control them in a uniform way, and so we are unable to obtain an asymptotic formula for $N_{\bZ}^\dagger(X)$ by summing over $f\in\fF$. \\

Finally, let us explain, for each $f\in\fF^*$, how counting forms in $V_\bZ^{\sm}\cap V_{\bZ,f}^0$ may be reduced to counting lattice points. Write $f(x,y) = \alpha x^2 + \beta xy + \gamma y^2$ with $\alpha,\beta,\gamma\in\bZ$. By (\ref{abc family}) and (\ref{abc family 2}), the set $V_{\bR,f}$ is a vector space isomorphic to $\bR^3$ via
\begin{align*}\Theta_1: a_4x^4 + a_3x^3y + a_2x^2y^2 + a_1xy^3 + a_0y^4&\mapsto (a_4,a_3,a_2)\hspace{1em}\mbox{if $\alpha\neq0$},\\\label{iso2}
\Theta_2: a_4x^4 + a_3x^3y + a_2x^2y^2 + a_1xy^3 + a_0y^4&\mapsto (a_4,a_2,a_0)\hspace{1em}\mbox{if $\beta,\beta^2+4\alpha\gamma\neq0$}.\end{align*}
Recall that the subset $V_{\bZ,f}$ has the structure of a rank-three $\bZ$-lattice, which may be identified with the lattices
\begin{equation}\label{Lambda def}\Lambda_{f,1} = \Theta_1(V_{\bZ,f})\AND\Lambda_{f,2}=\Theta_2(V_{\bZ,f})
\end{equation}
in $\bZ^3$. Let us mention here that we shall use the isomorphism 
\[\Theta_{w(f)},\mbox{ where }w(f)=\begin{cases}
1 & \mbox{if $f$ is irreducible},\\
2 &\mbox{if $f$ is reducible}.
\end{cases}\]
Thus, the problem is reduced to counting points in $\Lambda_{f,1}$ or $\Lambda_{f,2}$, and then sieving out those which come from reducible forms. In turn, counting lattice points amounts to computing certain volumes by a result of Davenport \cite{Dav}; see Proposition~\ref{Davenport}.

\subsection{Statement of the main theorem}

It is clear that we may choose the set $\fF$ of representatives to be such that for all $f\in\fF$, the $x^2$-coefficient is positive, and
\begin{equation}\label{reducible f shape}
f(x,y) = \alpha x^2 + \beta xy,\mbox{ where }\gcd(\alpha,\beta)=1\AND 0<\alpha\leq \beta
\end{equation}
when $f$ is reducible. Let $\sim$ denote $\GL_2(\bZ)$-equivalence. Then, our main result is:

\begin{theorem} \label{Small Gal MT} Let $f(x,y)$ be an integral and primitive binary quadratic form of non-zero discriminant and with positive $x^2$-coefficient. Write $D_f = |\Delta(f)|$, and put
\[s_f = \begin{cases}8&\text{if $D_f$ is odd},\\1&\text{if $D_f$ is even}.\end{cases}\]
\begin{enumerate}[(a)]
\item Suppose that $f$ is positive definite. Then, we have
\[\hspace{0.5cm}N^{(D_4)}_{\bZ,f}(X) = 
\dfrac{1}{s_fr_f}\dfrac{13 \pi }{27D_f^{3/2}} X^{3/2} + O_f(X^{1+\epsilon})\mbox{ for any }\epsilon>0,\]
where
\[\hspace{0.5cm}r_f = \begin{cases} 6 & \text{if $f(x,y) \sim x^2 + xy + y^2$},\\
2 & \mbox{if $f(x,y)\sim ax^2 + cy^2$}\\ &\hspace{1em}\mbox{ or $f(x,y)\sim ax^2 + bxy + ay^2$ with $a\neq b$},\\
1 & \text{otherwise}.\end{cases}\]
\item Suppose that $f$ is reducible and that $f$ has the shape (\ref{reducible f shape}). Then, we have
\[\hspace{0.5cm}N^{(D_4)}_{\bZ,f}(X)=
\dfrac{1}{s_fr_f}\dfrac{8}{9\beta^{3/2}}X^{3/2}\log X + O_f(X^{3/2}),\]
where
\[\hspace{0.5cm}r_f =\begin{cases}
1& \text{if $\beta\nmid\alpha^2+1$\AND $\beta\nmid \alpha^2-1$},\\
2& \text{otherwise}.
\end{cases}\]
\item Suppose that $f$ is indefinite and irreducible. Define $t_{D_f}\in\bR$ to be such that $e^{t_{D_f}}$ is the fundamental unit of the quadratic order  $\bZ[(D_f+\sqrt{D_f})/2]$, or equivalently
\[ t_{D_f} = \log((u_{D_f}+v_{D_f}\sqrt{D_f})/2),\]
where $(u_{D_f},v_{D_f})\in\bN^2$ is the least solution to $x^2-D_fy^2=\pm4$. Then, we have
\[\hspace{0.5cm}N_{\bZ,f}^{(D_4)}(X)=\frac{1}{s_fr_f}\frac{32t_{D_f}}{9D_f^{3/2}}X^{3/2}+ O_f(X^{1+\epsilon}) \mbox{ for any }\epsilon>0,\]
where 
\[ \hspace{0.5cm}r_f = \begin{cases}
2 & \mbox{if $f(x,y)\sim ax^2 + bxy - ay^2$}\\&\hspace{1em}\mbox{or $f(x,y)\sim ax^2 + b xy + cy^2$ with $a\mid b$},\\
1 & \text{otherwise}.\end{cases}\]
\item In all three cases, for any $\epsilon>0$, we have
\[ N_{\bZ,f}^{(V_4)}(X) = O_{f,\epsilon}(X^{1+\epsilon}),\]
and also
\[N_{\bZ,f}^{(C_4)}(X) = \begin{cases}O_{f,\epsilon}(X^{1/2+\epsilon})&\mbox{if $-\Delta(f)\neq\square$},\\O_f(X) &\mbox{if $-\Delta(f)=\square$}.\end{cases}\]
\end{enumerate}
\end{theorem}


Notice that the error terms in Theorem~\ref{Small Gal MT} depend upon $f$. Hence, we are unable to obtain an asymptotic formula for $N_{\bZ}^\dagger(X)$  by summing over $f\in\fF$. However, there are only three $f\in\fF$ that need to be considered if we restrict to the forms in
\[ V_{\bZ}^{\sm,*} = \{F \in V_{\bZ}^{\sm} : F_T = F\mbox{ for some }T\in\GL_2(\bZ)\setminus\{\pm I_{2\times 2}\}\}.\]
This is because by Proposition~\ref{auto theorem} below, such a matrix $T$ must be of the shape $M_f$ or $M_f/2$ up to sign, where $f\in\fF^*$. From (\ref{union}), we then deduce that
\begin{align*} V_{\bZ}^{\sm,*} &= \bigcup_{\substack{f\in\fF^*\\\Delta(f)\in\{-4,1,4\}}} \{F\in V_{\bZ,f}^0:F\mbox{ is irreducible}\},\\\notag
V_{\bZ}^{\sm,*,\dagger} &= \bigsqcup_{\substack{f\in\fF^*\\\Delta(f)\in\{-4,1,4\}}} \{F\in V_{\bZ,f}^0:F\mbox{ is irreducible}\AND \Gal(F)\not\simeq V_4\}.\end{align*}
For $X>0$, let us put
\[ N_{\bZ}^{*,\dagger}(X) = \#\{[F]: F \in V_{\bZ}^{\sm,*,\dagger}\mbox{ such that }H(F)\leq X\}.\]
Then, by (\ref{V bijection}) and the above discussion, we have
\[ N_{\bZ}^{*,\dagger}(X) = N_{\bZ,f^{(1)}}^{*,\dagger}(X)+N_{\bZ,f^{(2)}}^{*,\dagger}(X)+N_{\bZ,f^{(3)}}^{*,\dagger}(X), \]
where we may take
\[ f^{(1)}(x,y) = x^2+y^2,\, f^{(2)}(x,y) = x^2 + xy,\, f^{(3)}(x,y) = x^2 +2xy,\]
whose discriminants are $-4,1$, and $4$, respectively. It follows that:

\begin{corollary}\label{corollary}We have
\[ N_{\bZ}^{*,\dagger}(X) = \frac{1}{9}X^{3/2}\log X + O(X^{3/2}).\]
\end{corollary}
\begin{proof}Theorem~\ref{Small Gal MT} implies that 
\[N_{\bZ,f^{(1)}}^\dagger(X) = O(X^{3/2})\AND
N_{\bZ,f^{(i)}}^\dagger(X) = \frac{1}{18}X^{3/2}\log X + O(X^{3/2})\mbox{ for $i=2,3$}\]
Summing these terms up then yields the claim.
\end{proof}

Finally, as a consequence of the proof of Theorem~\ref{Small Gal MT}, we also have:

\begin{theorem} \label{negative Pell}Let $D = \beta^2 + 4\alpha^2$, where $\alpha,\beta\in\bN$ are coprime and $D$ is not a square. Then, the negative Pell's equation $x^2 - Dy^2 = -4$ has integer solutions if and only if the integral binary quadratic form $\alpha x^2 + \beta xy  - \alpha y^2$ is $\GL_2(\bZ)$-equivalent to a form of the shape $a x^2 + b xy + c y^2$ with $a$ dividing $b$.
\end{theorem}

We now discuss some potential appli\-cations of our Theorem~\ref{Small Gal MT} and Corollary~\ref{corollary}.\\

First, it is natural to ask whether the asymptotic formula (\ref{BS formula}), which was proven using Proposition~\ref{Davenport}, admits a secondary main term. From the arguments in \cite{BhaSha}, we see that the error term arising from volumes of the lower dimensional projections in Proposition~\ref{Davenport} is only of order $O(X^{3/4})$. Thus, possibly $X^{3/4}$ is the order of a second main term, but it is dominated by another error term coming from
\[ N_{\bZ,\BS}^{*}(X) = \#\{[F]: F \in V_{\bZ}^{\sm,*}\mbox{ such that }H_\BS(F)\leq X\}.\]
In particular, it was shown in \cite[Lemma 2.4]{BhaSha} that
\[ N_{\bZ,\BS}^*(X) = O_{\epsilon}(X^{3/4+\epsilon})\mbox{ for any }\epsilon>0.\]
Our Corollary~\ref{corollary} removes this obstacle, because
\[ N_{\bZ}^{*,\dagger}(X^{1/3})  \leq N_{\bZ,\BS}^*(X) \leq N_{\bZ}^{*,\dagger}(10X^{1/3}) + O_{\epsilon}(X^{1/3+\epsilon})\]
by (\ref{H compare}) and Theorem~\ref{Small Gal MT} (d), whence we have
\[N_{\bZ,\BS}^*(X)\asymp X^{1/2}\log X.\]
This improvement potentially allows one to prove a secondary main term for (\ref{BS formula}) by using similar methods from \cite{BhaShaTsi}, where it was shown that the counting theorem in \cite{DH} for cubic fields has a secondary main term of order $X^{5/6}$; this latter fact was proven independently in \cite{TT} as well.\\

Next, integral binary quartic forms are closely related to quartic orders, and maximal irreducible quartic orders may be regarded as quartic fields. More generally, by the construction of Birch-Merriman \cite{BM} or Nakagawa \cite{Nakagawa}, any integral binary form $F$ gives rise to a $\bZ$-order $Q_F$ whose rank is the degree of $F$, where $\GL_2(\bZ)$-equivalence class of $F$ corresponds to isomorphism class of $Q_F$. By \cite{DF}, it is well-known that all cubic orders come from integral binary cubic forms, which enabled the enumeration of cubic orders having a non-trivial automorphism as well as cubic fields by their discriminant; see \cite{BhaShn} and \cite{DH}, respectively. But this is not true for orders of higher rank. Parametrizations of quartic and quintic orders were given by Bhargava in his seminal work \cite{HCL3} and \cite{HCL4}. In \cite{Wood}, Wood further showed that the quartic orders arising from integral binary quartic forms are exactly those having a monogenic \emph{cubic resolvent}; see \cite{HCL3} for the definition. This implies that the forms in
\[V_{\bZ}^{\sm,\star}= \{F\in V_\bZ^{\sm}: Q_F\mbox{ is maximal}\}\]
correspond to quartic $D_4$-, $C_4$-, and $V_4$-fields whose ring of integers has a monogenic cubic resolvent. In our upcoming paper \cite{TX2}, we shall enumerate $\GL_2(\bZ)$-equivalence classes of forms in $V_{\bZ}^{\sm,\star}$ with respect to a height corresponding to the conductor of fields, as motivated by \cite{ASVW}. In fact, we shall that show that 
\[\mbox{for all }f\in\fF^*:F\in V_{\bZ}^{\sm,\star}\cap V_{\bZ,f}^0 \neq\emptyset \mbox{ if and only if }\Delta(f) \in \{-4,1,4\}.\]
Thus, our counting theorem in \cite{TX2} may be regarded as a refinement and an extension of Corollary~\ref{corollary} above. \\

Last but not least, binary quartic forms are connected to elliptic curves as well. In particular, any integral binary quartic form $F$ gives rise to an elliptic curve 
\[E_F : y^2 = x^3 - \frac{I(F)}{3}x - \frac{J(F)}{27}\]
defined over $\bQ$. In \cite{BhaSha}, Bhargava and Shankar applied (\ref{BS formula}) as well as a parametrization of 2-Selmer groups due to Birch and Swinnerton-Dyer to show that the average rank of elliptic curves over $\bQ$, when ordered by a \emph{naive} height analogous to (\ref{BS height}), is at most $3/2$. This result is remarkable in that it is the first to show, unconditional on the BSD-conjecture  and the Grand Riemann Hypothesis, boundedness of the average rank of large families of elliptic curves over $\bQ$. Conditional bounds were obtained by Brumer \cite{Bru}, Heath-Brown \cite{HB}, and Young \cite{You} previously. Now, the relations in (\ref{IJ family}) imply that  for $F\in V_{\bZ}^{\sm}\cap V_{\bZ,f}^0$ with $f\in\fF^*$, we have
\[  E_F: y^2 = \left(x+\frac{L_f(F)}{3}\right)\left(x^2 - \frac{L_f(F)}{3}x - \frac{K_f(F)}{9}\right),\]
which has a rational $2$-torsion point. Hence, our Theorem~\ref{Small Gal MT} potentially allows one to study arithmetic properties of elliptic curves with $2$-torsion over $\bQ$. Let us remark that unlike a \emph{large} family of elliptic curves over $\bQ$, in the sense of \cite[Section 3]{BhaSha}, the family consisting of those curves with a rational $2$-torsion exhibits a rather peculiar behaviour. Indeed, Klagsbrun and Lemke-Oliver \cite{K-LO} proved that the average size of the 2-Selmer groups in this family is unbounded, and they conjectured an asymptotic growth rate. One might be able to obtain such an asymptotic growth rate using our Theorem \ref{Small Gal MT} and a sieve that detects local solubility; this line of inquiry is pursued in an upcoming paper due to D.~Kane and Z.~Klagsbrun. 


\section{Characterization of forms with small Galois groups}\label{Galois gp sec}

\subsection{Cremona covariants}\label{Cremona section}

Let $F$ be a real binary quartic form with $\Delta(F)\neq0$. As Cremona defined in \cite{Cre}, we have three quadratic covariants $\fC_{F,\omega}(x,y)$, each of which is associated to a root $\omega$ of $\Q_F(x)$; see \cite[Subsection 4.2]{X} for the explicit definition. They satisfy the syzygy
\begin{equation} \label{cre square} \fC_{F,\omega}(x,y)^2=\frac{1}{3} \left(F_4(x,y) + 4 \omega F(x,y)\right),\end{equation}
where $F_4$ is the \emph{Hessian covariant of $F$} and is given by
\begin{align*}
F_4(x,y) & = 3(a_3^2-8a_4a_2)x^4 + 4(a_3a_2-6a_4a_1)x^3y + 2(2a_2^2 - 24a_4a_0 - 3a_3a_1)x^2y^2 \\ & \hspace{6.25cm}+ 4(a_2a_1 - 6a_3a_0)xy^3 + (3a_1^2-8a_2a_0)y^4.
\end{align*}
We shall label the roots $\omega_1(F),\omega_2(F),\omega_3(F)$ of $\Q_F(x)$ such that 
\[ \fC_{F,\omega_i(F)}(x,y) = \fC_{F,i}(x,y)\mbox{ for all }i=1,2,3,\]
where $\fC_{F,i}(x,y)$ is defined as in \cite[(4.6)]{X}. Then, from (\ref{cre square}) and the explicit expressions for $\fC_{F,\omega}(x,y)$ given in \cite{X}, we have the following observations:
\begin{enumerate}[(1)]
\item For $\omega = \omega_1(F)$, the binary quadratic form $\fC_{F,\omega}(x,y)$ has real coefficients.
\item For $\omega = \omega_2(F),\omega_3(F)$, we have:
\\$\bullet$ If $\Delta(F)>0$, then $\lambda_\omega\cdot\fC_{F,\omega}(x,y)$ has real coefficients for some $\lambda_\omega\in\{1,\sqrt{-1}\}$.
\\$\bullet$  If $\Delta(F)<0$, then $\lambda\cdot\fC_{F,\omega}(x,y)$ does not have real coefficients for all $\lambda\in\bC^\times$.
\end{enumerate}
Also, it is easy to check that
\begin{equation}\label{cre disc}
\Delta(\fC_{F,\omega_1(F)}), \Delta(\fC_{F,\omega_3(F)}) > 0 \AND \Delta(\fC_{F,\omega_2(F)})<0.
\end{equation}
We shall require the following result by the second-named author in \cite{X}.

\begin{proposition} \label{auto theorem} Let $F$ be a real binary quartic form with $\Delta(F)\neq0$. Then, a set of representatives for the quotient group
\[\{T\in\GL_2(\bR): F_T = F\}/\{\lambda\cdot I_{2\times2}:\lambda\in\bR^\times\}\]
is given by
\[ \begin{cases}\{I_{2\times2}, M_{f}: f \in\{\fC_{F,\omega_1(F)},\lambda_{\omega_2(F)}\cdot\fC_{F,\omega_2(F)},\lambda_{\omega_3(F)}\cdot\fC_{F,\omega_3(F)} \}&\text{if } \Delta(F) > 0,\\
\{I_{2\times 2},M_{f}: f \in\{\fC_{F,\omega_1(F)}\}\}&\text{if } \Delta(F) < 0.\end{cases}\]
Furthermore, the quadratic forms $\fC_{F,\omega_1(F)}(x,y),\fC_{F,\omega_2(F)}(x,y)$, and $\fC_{F,\omega_3(F)}(x,y)$, are pairwise non-proportional over $\bC^\times$.
\end{proposition}
\begin{proof}For the first statement, see \cite[Proposition 4.6]{X}. As for the second statement, since $\fC_{F,\omega_i(F)}(x,y)$ are covariants, replacing $F$ by a $\GL_2(\bR)$-translate if necessary, we may assume that $F(x,y) = a_4x^4 + a_2x^2y^2 \pm a_4y^4$. In this special case, it is not hard to verify the claim using the explicit expressions for $\fC_{F,\omega_i(F)}(x,y)$ in \cite[(4.6)]{X}.
\end{proof}

Let $F$ be a real binary quartic form with $\Delta(F)\neq0$. Proposition~\ref{auto theorem} implies that for any real binary quadratic form $f$ with $\Delta(f)\neq0$, we have $F\in V_{\bR,f}$ if and only if
\begin{equation}\label{in V char}
f(x,y)\mbox{ is proportional to $\fC_{F,\omega}(x,y)$ for a root $\omega$ of $\Q_F(x)$}.
\end{equation}
Moreover, this root $\omega$ is unique, and we shall denote it by $\omega_f(F)$.  This was required in order to define the $L_f$- and $K_f$-invariants in (\ref{LK def}).

\subsection{Proof of Theorem~\ref{small char thm}}

The key is the following lemma.

\begin{lemma}\label{omega in Z} Let $F$ be an integral binary quartic form with $\Delta(F)\neq0$ and let $\omega$ be a root of $\Q_F(x)$. Then, the quadratic form $\fC_{F,\omega}(x,y)$ is proportional over $\bC^\times$ to a form with integer coefficients if and only if $\omega\in\bZ$.
\end{lemma}
\begin{proof}If $\omega\in\bZ$, then we easily see from (\ref{cre square}) that $\lambda\cdot \fC_{F,\omega}(x,y)$ has integer coefficients for some $\lambda\in\bC^\times$. Conversely, if $\lambda\cdot \fC_{F,\omega}(x,y)$ has integer coefficients for some $\lambda\in\bC^\times$, then consider the action of an element $\sigma\in\Gal(\overline{\bQ}/\bQ)$, where $\overline{\bQ}$ is an algebraic closure of $\bQ$. It is clear from the definition of $ \fC_{F,\omega}(x,y)$ that $\lambda\in\overline{\bQ}$. From (\ref{cre square}), we have
\[ \frac{4}{3}(\omega - \sigma(\omega))F(x,y) = \fC_{F,\omega}(x,y)^2 - \sigma(\fC_{F,\omega}(x,y)^2) =\left(1-\frac{\lambda^2}{\sigma(\lambda)^2}\right)\fC_{F,\omega}(x,y)^2,\]
and this last binary quartic form has zero discriminant. This shows that $\omega - \sigma(\omega) = 0$ for all $\sigma\in\Gal(\overline{\bQ}/\bQ)$. Thus, we have $\omega\in\bQ$, and so $\omega\in\bZ$ since $\Q_F(x)$ is monic.
\end{proof}

The first claim in Theorem~\ref{small char thm} now follows from Proposition~\ref{auto theorem}, Lemma~\ref{omega in Z}, and (\ref{in V char}). Note that $\Delta(F) = 27^2\Delta(\Q_F)$, which means that $\Q_F(x)$ has three integer roots if and only if $\Q_F(x)$ is reducible and $\Delta(F)=\square$. The second claim then follows from this fact and (\ref{cre disc}).

\section{Basic properties of forms in $V_{\bR,f}$ of non-zero discriminant}
\label{properties section}

Throughout this section, let $f(x,y) = \alpha x^2 + \beta xy + \gamma y^2$ be a real binary quadratic form with $\Delta(f)\neq0$. It is not hard to check, by a direct calculation, that
\begin{equation}\label{abc family}
 V_{\bR,f} =\left\lbrace
\begin{array}{@{}c@{}c}
Ax^4 + Bx^3y + Cx^2 y^2 + \left(\dfrac{4 \beta \gamma A - (\beta^2 + 2 \alpha \gamma) B + 2 \alpha \beta C}{2\alpha^2} \right)xy^3 \\\\
+\left(\dfrac{4 \gamma(\beta^2 + 2\alpha \gamma) A - \beta(\beta^2 + 4 \alpha \gamma) B + 2\alpha \beta^2 C}{8\alpha^3}\right)y^4:
A,B,C\in\bR
\end{array}
\right\rbrace
\end{equation}
if $\alpha\neq0$, and similarly that
\begin{equation}\label{abc family 2}
 V_{\bR,f} =\left\lbrace
\begin{array}{@{}c@{}c}
Ax^4 +\left( \dfrac{\gamma(4 \beta^2 + 8 \alpha \gamma)A + 2 \alpha \beta^2 B - 8 \alpha^3 C}{\beta(\beta^2 + 4 \alpha \gamma)} \right)x^3 y + B x^2 y^2\\\\
- \left(\dfrac{8 \gamma^3 A - 2 \beta^2 \gamma B - \alpha(4  \beta^2 + 8 \alpha \gamma)C}{\beta(\beta^2 + 4 \alpha \gamma)} \right) xy^3 + Cy^4:
A,B,C\in\bR
\end{array}
\right\rbrace
\end{equation}
if $\beta,\beta^2+4\alpha\gamma\neq0$. Below, we shall give some basic properties of $V_{\bR,f}^0$ and $V_{\bZ,f}^0$. 

\subsection{The two new invariants}
\label{LK section}

Recall the definitions of the $L_f$- and $K_f$-invariants given in (\ref{LK def}). First, we shall show that they are indeed invariants under the twisted action of $\GL_2(\bR)$ in the following sense.

\begin{proposition}\label{LK invariant}For all $F\in V_{\bR,f}^0$ and $T\in\GL_2(\bR)$, we have 
\[L_{f_T}(F_T) = L_f(F)\AND K_{f_T}(F_T) = K_f(F).\]
\end{proposition}
\begin{proof}Notice that $\Q_F(x) = \Q_{F_T}(x)$. For any root $\omega$ of $\Q_F(x)$, because $\fC_{F,\omega}(x,y)$ is a covariant up to sign by (\ref{cre square}), if $\fC_{F,\omega}(x,y)$ is proportional to $f(x,y)$, then $\fC_{F_T,\omega}(x,y)$ is proportional to $f_T(x,y)$. It then follows from the definition that $ L_{f_T}(F_T) = L_f(F)$. Since $I(F_T) = I(F)$, we also have $K_{f_T}(F_T) = K_f(F)$ by the first equality in (\ref{IJ family}).
\end{proof}

We shall give explicit formulae for $L_f(-)$ and $K_f(-)$ in two special cases.

\begin{proposition}\label{explicit LK}The following holds.
\begin{enumerate}[(a)]
\item Assume that $\alpha\neq0$. Then, for all $F\in V_{\bR,f}^0$ as in (\ref{abc family}), we have
\begin{align*}
\hspace{5mm}  L_f(F) &= -(12 \gamma A - 3 \beta B + 2 \alpha C)/(2\alpha),\\
\hspace{5mm} K_f(F)&= (72 \beta^2 \gamma A^2 + 9 \alpha(\beta^2+4 \alpha \gamma)B^2 + 8 \alpha^3 C^2\\\notag
&\hspace{2em}- 18 \beta(\beta^2+4 \alpha \gamma)AB + 12 \alpha(3\beta^2-4 \alpha \gamma)AC - 24 \alpha^2 \beta BC)/(4
\alpha^3).\end{align*}
Moreover, we have
\[\hspace{5mm} \frac{4(L_f(F)^2 + 4K_f(F))}{9} = \frac{L_{f,1}(F)^2 - \Delta(f)L_{f,2}(F)^2}{\alpha^4},\]
where
\[\hspace{5mm} L_{f,1}(F)  = 4(\beta^2 - \alpha\gamma)A - 3\alpha\beta B + 2\alpha^2 C \AND L_{f,2}(F)  = 2(2\beta A - \alpha B). \]
\item Assume that $\gamma=0$. Then, for all $F\in V_{\bR,f}^0$ as in (\ref{abc family 2}), we have
\begin{align*} 
\hspace{5mm} L_f(F) & = (2\beta^2B - 12\alpha^2C)/\beta^2,\\
\hspace{5mm} K_f(F) & = (-\beta^4 B^2 + 144\alpha^4C^2 + 36\beta^4AC - 24\alpha^2\beta^2BC)/\beta^4. \end{align*}
Moreover, we have
\[\hspace{5mm} \frac{4(L_f(F)^2 + 4K_f(F))}{9}=\frac{8C}{\beta^2}\left(8\beta^2A-8\alpha^2B+\frac{40\alpha^4}{\beta^2}C\right).\]
\end{enumerate}
\end{proposition}
\begin{proof}This may be verified by explicit computation.
\end{proof}

We shall also need the following observation.

\begin{proposition}\label{LK integers}Assume that $f$ is integral. Then, for all $F\in V_{\bZ,f}^0$, we have 
\[L_f(F), K_f(F),(L_f(F)^2 + 4K_f(F))/9, (2L_f(F)^2 - K_f(F))/9\in\bZ.\]
Moreover, when $f$ is primitive in addition, we have
\[ 4(2L_f(F)^2 - K_f(F))/(9\Delta(f)) \in \bZ.\]
\end{proposition}
\begin{proof}We have $L_f(F)\in\bZ$ by Lemma~\ref{omega in Z}. Since $I(F)\in\bZ$, we deduce from the first equality in (\ref{IJ family}) that $K_f(F)\in\bZ$ holds as well. Observe that
\begin{align*} I(F)+K_f(F)&=(L_f(F)^2+4K_f(F))/3,\\
2I(F)-K_f(F)&=(2L_f(F)^2-K_f(F))/3,\end{align*}
both of which are integers. Since $\Delta(F)\in\bZ$, we deduce from (\ref{Delta LK}) that at least one of the above expressions is divisible by $3$. But again by (\ref{IJ family}), we have 
\[3I(F)=(L_f(F)^2+4K_f(F))/3+(2L_f(F)^2-K_f(F))/3,\]
so in fact both expressions are divisible by $3$. This proves the first claim.\\

Next, assume that $f$ is primitive in addition. In view of Proposition~\ref{LK invariant}, by applying a $\GL_2(\bZ)$-action on $f$ if necessary, we may assume that $\alpha\neq0$ and that $\alpha$ is coprime to $\Delta(f)$. Using Proposition~\ref{explicit LK} (a), we then compute that
\[\label{2LKD} \frac{4(2L_f(F)^2 - K_f(F))}{9} = \Delta(f)\left(\frac{ \alpha(B^2-4AC) + 2A(\beta B - 4\gamma A)}{\alpha^3}\right).\]
This expression is an integer by the first claim, and hence must be divisible by $\Delta(f)$, because $\alpha$ is taken to be coprime to $\Delta(f)$. This proves the second claim. 
\end{proof}

\subsection{Determinants of the two lattices}\label{det sec}

In this subsection, assume that $f$ is integral and primitive. Let $\Lambda_{f,1}$ and $\Lambda_{f,2}$ denote the lattices defined in (\ref{Lambda def}). Below, we shall compute their determinants in terms of the number $s_f$ as in Theorem~\ref{Small Gal MT}.

\begin{proposition}\label{det prop}We have $\det(\Lambda_{f,1}) = s_f|\alpha|^3$ and $\det(\Lambda_{f,2}) = s_f|\beta(\beta^2+4\alpha\gamma)|/8$.
\end{proposition}
\begin{proof}Observe that the linear transformation defined by the matrix
\[\begin{pmatrix}1&0&0\\0&0&1\\ *&-\B&*\end{pmatrix},\text{ where }\B = \frac{\beta(\beta^2+4\alpha\gamma)}{8\alpha^3},\]
has determinant $\B$, and it sends $\Lambda_{f,1}$ to $\Lambda_{f,2}$. Thus, it suffices to prove the first claim. Recall from (\ref{abc family}) that $\Lambda_{f,1}$ is the set of tuples $(A,B,C)\in\bZ^3$ satisfying
\begin{align*}
4\beta\gamma A-(\beta^2+2\alpha\gamma)B+2\alpha\beta C&\equiv0\hspace{-2cm}&\pmod{2\alpha^2},\\
4\gamma(\beta^2+2\alpha\gamma)A-\beta(\beta^2+4\alpha\gamma)B+2\alpha\beta^2C&\equiv 0\hspace{-2cm}&\pmod{8\alpha^3}.
\end{align*}
If $\beta\gamma=0$, then it is easy to check that $\det(\Lambda_{f,1})=s_f|\alpha|^3$. If $\beta\gamma\neq0$, then we shall use the fact that
\[ \det(\Lambda_{f,1}) = \prod_{p}\det(\Lambda_{f,1}^{(p)})= \prod_{p\mid 2\alpha}\det(\Lambda_{f,1}^{(p)}),\mbox{ where }\Lambda_{f,1}^{(p)} = \bZ_p\otimes_\bZ\Lambda_{f,1},\]
and so $\det(\Lambda_{f,1})=s_f|\alpha|^3$ indeed holds by Lemma~\ref{Lap det} below.
\end{proof}

\begin{lemma}\label{Lap det}Let $p$ be a prime dividing $2\alpha$ and let $p^k\| \alpha$. Then, we have
\[ \det(\Lambda_{f,1}^{(p)}) = s_f^{\epsilon_p}p^{3k}, \mbox{ where }\epsilon_p = \begin{cases} 1 &\mbox{if $p=2$},\\ 0 & \mbox{if $p\geq 3$}.\end{cases}\]\end{lemma}
\begin{proof}For brevity, write
\[\alpha = p^k a\AND \beta = p^\ell b,\mbox{ where $k,\ell,a,b\in\bZ$ with $k,\ell\geq0$ and $p\nmid a,b$}.\]
Then, the claim may be restated as 
\[ \det(\Lambda_{f,1}^{(p)}) = \begin{cases} p^{3k+3\epsilon_p} &\mbox{if $\ell=0$},\\  p^{3k} & \mbox{if $\ell\geq1$}.\end{cases}\]
By definition, the lattice $\Lambda_{f,1}^{(p)}$ is the set $(A,B,C)\in\bZ_p^3$ of tuples satisfying
\[\T_1(A,B,C)\equiv0\hspace{-3mm}\pmod{p^{2k+\epsilon_p}}\AND \T_2(A,B,C)\equiv0\hspace{-3mm}\pmod{p^{3k+3\epsilon_p}},\]
where
\begin{align*}\label{A1 A2 p}
\T_1(A,B,C)&=p^\ell b(4\gamma A-p^\ell bB)-2p^ka\gamma B+2p^{k+\ell}abC,\\\notag
\T_ 2(A,B,C) &=(p^{2\ell}b^2+4p^ka\gamma)(4\gamma A-p^\ell bB) - 8p^ka\gamma^2 A+2p^{k+2\ell}ab^2C.\end{align*}
Observe that we have the relation
\begin{equation}\label{A1 A2 p relation}
\T_2(A,B,C)-p^\ell b\T_1(A,B,C)=2p^ka\gamma(4\gamma A-p^\ell bB).
\end{equation}
For $\ell=0$, we deduce from (\ref{A1 A2 p relation}) that $\Lambda_{f,1}^{(p)}$ is defined solely by 
\[\T_2(A,B,C)\equiv0\hspace{-3mm}\pmod{p^{3k+3\epsilon_p}}.\]
For $\ell\geq1$ and $\ell\geq k+2\epsilon_{p}$, it is easy to see that $\Lambda_{f,1}^{(p)}$ is in fact defined by 
\[A\equiv0\mbox{ (mod $p^{2k}$) and }B\equiv0\mbox{ (mod $p^k$)}.\]
For $\ell\geq1$ and $\ell\leq k+\epsilon_p$, we shall first show that $\Lambda_{f,1}^{(p)}$ is also defined by
\begin{equation}\label{congruences}\begin{cases}
A\equiv0&\pmod{p^{2\ell-2\epsilon_p}},\\
B\equiv0&\pmod{p^{\ell-\epsilon_p}},\\
(4\gamma A-p^\ell bB)/p^{2\ell-\epsilon_p}\equiv0&\pmod{p^{k-\ell+\epsilon_p}},\\
\T_2(A,B,C)/p^{k+2\ell+\epsilon_p}\equiv 0&\pmod{p^{2k-2\ell+2\epsilon_p}}.\end{cases}\end{equation}
If (\ref{congruences}) is satisfied, then from (\ref{A1 A2 p relation}), it is easy to see that $(A,B,C)\in\Lambda_{f,1}^{(p)}$. Conversely, if $(A,B,C)\in\Lambda_{f,1}^{(p)}$, then the assumption $\ell\leq k+\epsilon_p$ implies that
\[ \mbox{$\T_1(A,B,C)\equiv0$ (mod $p^{k+\ell}$)} \AND \mbox{$\T_2(A,B,C)\equiv0$ (mod $p^{k+2\ell+\epsilon_p}$)},\]
while reducing (\ref{A1 A2 p relation}) mod $p^{2k+\ell+\epsilon_p}$ also yields
\[ 4\gamma A-p^\ell bB\equiv0\mbox{ (mod $p^{k+\ell}$)}.\]
From these three congruence equations, it follows that (\ref{congruences}) is indeed satisfied. In all cases, we then see that $\det(\Lambda_{f,1}^{(p)})$ is as claimed.
\end{proof}


\subsection{Forms with abelian Galois groups}\label{abelian Gal sec}

In this subsection, assume that $f$ is inte\-gral. Consider an irreducible form $F\in V_{\bZ,f}^0$. By Theorem~\ref{small char thm}, we have $\Gal(F)\simeq D_4$, $C_4$, or $V_4$. To distinguish among these three possibilities, note that the \emph{cubic resolvent polynomial of $F$}, defined by
\[ R_F(x) = a_4^3X^3 - a_4^2a_2X^2 + a_4(a_3a_1 - 4a_4a_0)X - (a_3^2a_0 + a_4a_1^2 - 4a_4a_2a_0)\]
when $F$ has the shape (\ref{F generic}), is reducible since $\Gal(F)$ is small. Also, it has a unique root $r_F\in\bQ$ precisely when $\Delta(F)\neq\square$, in which case we define
\[\theta_1(F) = (a_3^2 - 4a_4(a_2 - r_Fa_4))\Delta(F)\AND \theta_2(F) = a_4(r_F^2a_4 - 4a_0)\Delta(F).\]
Then, we have the well-known criterion
\begin{align*}
\Gal(F)\simeq V_4 & \iff \Delta(F)=\square,\\
\Gal(F)\simeq C_4 & \iff \Delta(F) \neq\square \mbox{ and }\theta_1(F),\theta_2(F)=\square\mbox{ in $\bQ$}.
\end{align*}
See \cite{Con} for example. We then deduce that:

\begin{proposition}\label{abelian Gal prop}Let $F\in V_{\bZ,f}^0$ be an irreducible form. Then, we have
\[\Gal(F)\simeq V_4\iff L_f(F)^2 + 4K_f(F)=\square,\]
as well as
\[ \Gal(F)\simeq C_4\iff \begin{cases}L_f(F)^2+4K_f(F)\neq\square,\\(L_f(F)^2 + 4K_f(F))(2L_f(F)^2 - K_f(F))/\Delta(f)=\square.\end{cases}\]
\end{proposition}
\begin{proof}Observe that by (\ref{Delta LK}), we have
\[ \Delta(F)=\square\mbox{ if and only if }L_f(F)^2+4K_f(F)=\square.\]
The first claim is then clear. Next, suppose that $\Delta(F)\neq\square$. By Proposition~\ref{LK invariant}, we may assume that $\alpha\neq0$. For $F$ in the shape as in (\ref{abc family}), a direct computation yields
\[ r_F = (-4\gamma A + \beta B)/(2\alpha A).\]
Using Proposition~\ref{explicit LK} (a), we further compute that
\begin{align*}
\theta_1(F) & = 4\alpha^2(2L_f(F)^2 - K_f(F))\Delta(F)/(9 \Delta(f)),\\
\theta_2(F) & = \beta^2 (2L_f(F)^2-K_f(F))\Delta(F)/(9 \Delta(f)).
\end{align*}
By (\ref{Delta LK}) and the criterion above, it follows that $\theta_1(F), \theta_2(F)$ are squares if and only if $(L_f(F)^2 + 4K_f(F))(2L_f(F)^2 - K_f(F))/\Delta(f)$ is a square, as desired. 
\end{proof}

\subsection{Reducible forms} In this subsection, assume that $f$ is integral. We shall study the reducible forms in $V_{\bZ,f}^0$. Let us first make a definition and an observation.

\begin{definition}\label{reducible types def}Let $F\in V_{\bZ,f}^0$ be a reducible form.
\begin{enumerate}[(1)]
\item We say that $F$ is of \emph{type $1$} if $F = m\cdot pp_{M_f}$ for some $m\in\bQ^\times$ and integral binary quadratic form $p$.
\item We say that $F$ is of \emph{type $2$} if $F = pq$ for some integral binary quadratic forms $p$ and $q$ satisfying $p_{M_f} = -p$ and $q_{M_f} = - q$.
\end{enumerate}
\end{definition}

\begin{lemma}\label{type 1}For all reducible forms $F\in V_{\bZ,f}^0$ of type $1$, we have
\[ L_f(F)^2 + 4K_f(F) = \square.\]
\end{lemma}
\begin{proof}This may be verified by a direct computation.
\end{proof}

Below, we shall show that the two reducibility types in Definition~\ref{reducible types def} are in fact the only possibilities. We shall require two further lemmas.

\begin{lemma}\label{linear factor}Let $\ell(x,y) = \ell_1x+\ell_0y$ be a non-zero complex binary linear form, and suppose that $\ell_{M_f} =\lambda\cdot \ell$ for some $\lambda\in\bC^\times$. Then, we have $\lambda = \pm\sqrt{-1}$, with
\[\lambda = \begin{cases}
-\sqrt{-1} &\mbox{if and only if } \ell_0 = (\beta +\sqrt{\Delta(f)})\ell_1/(2\alpha),\\
\sqrt{-1} &\mbox{if and only if } \ell_0 = (\beta - \sqrt{\Delta(f)})\ell_1/(2\alpha),
\end{cases}\]
in the case that $\alpha\neq0$.
\end{lemma}
\begin{proof}The hypothesis implies that
\[ \frac{1}{\sqrt{-\Delta(f)}}\begin{pmatrix} \beta & - 2\alpha \\ 2\gamma & -\beta\end{pmatrix}\begin{pmatrix}\ell_1 \\ \ell_0\end{pmatrix} = 
\lambda\begin{pmatrix} \ell_1 \\ \ell_0\end{pmatrix}.\]
Then, by computing the eigenvalues and eigenspaces of the $2\times2$ matrix above, we see that the claim holds.
\end{proof}

\begin{lemma}\label{quadratic factor}Let $p(x,y) = p_2 x^2 + p_1xy + p_0y^2$ be a non-zero complex binary quadratic form, and suppose that $p_{M_f} = \lambda\cdot p$ for some $\lambda\in\bC^\times$. Then, we have $\lambda = \pm1$, with
\[\lambda = \begin{cases}-1&\text{if and only if  $p_0 = (\beta p_1 - 2\gamma p_2)/(2\alpha)$},\\ 
1& \text{if and only if $p = (p_2/\alpha)f$},
\end{cases}\]
in the case that $\alpha\neq0$.
\end{lemma}

\begin{proof}The hypothesis implies that
\[ \frac{1}{-\Delta(f)}\begin{pmatrix} 
\beta^2 & -2\alpha & 4\alpha^2 \\ 4\beta\gamma & -(\beta^2+4\alpha\gamma) & 4\alpha\beta \\ 4\gamma^2 & -2\beta\gamma & \beta^2
\end{pmatrix}\begin{pmatrix} p_2 \\ p_1 \\ p_0\end{pmatrix} = \lambda\begin{pmatrix} p_2 \\ p_1 \\ p_0\end{pmatrix}.\]
Then, by computing the eigenvalues and eigenspaces of the $3\times3$ matrix above, it is not hard to check that the claim holds.\end{proof}



\begin{proposition}\label{reducibility types}Any reducible form $F\in V_{\bZ,f}^0$ is either of type $1$ or of type $2$.
\end{proposition}

\begin{proof}Write $F = g^{(1)}g^{(2)}g^{(3)}g^{(4)}$, where the $g^{(k)}$ are complex binary linear forms, and are pairwise non-proportional because $\Delta(F)\neq0$. Since $F$ is reducible, by renumbering if necessary, we may assume that
\[\begin{cases}
g^{(1)}, g^{(2)}g^{(3)}g^{(4)} &\mbox{when $F$ has exactly one rational linear factor},\\
g^{(1)},g^{(2)}, g^{(3)}g^{(4)} &\mbox{when $F$ has exactly two rational linear factors},\\
g^{(1)}g^{(2)},g^{(3)}g^{(4)} &\mbox{when $F$ has no rational linear factor},\\
g^{(1)}, g^{(2)},g^{(3)},g^{(4)} &\mbox{when $F$ has four rational linear factors},
\end{cases}\]
have integer coefficients and are irreducible. We have $M_f^2 = \Delta(f)\cdot I_{2\times2}$ and $F_{M_f}= F$ by definition. Hence, up to scaling, the matrix $M_f$ acts on the $g^{(k)}$ via a permutation $\sigma$ on four letters of order dividing two. This has two consequences.\\

By (\ref{V bijection}), without loss of generality, we may assume that $\alpha\neq0$. First, the form $F$ cannot have exactly one rational linear factor, for otherwise
\[ \sigma(1) = 1 \AND \sigma(k_0) = k_0 \mbox{ for at least one $k_0\in\{2,3,4\}$}.\]
From Lemma~\ref{linear factor}, it would follow that $\Delta(f)$ is a square and that $g^{(k_0)}$ is proportional to a form with integer coefficients, which is a contradiction. Second, when $F$ has four rational linear factors, by further renumbering if necessary, we may assume that 
\[\sigma\in\{(1), (12), (12)(34)\}.\]
Now, in all three of the possible cases for the factorization of $F$, define
\[ p = g^{(1)}g^{(2)}\AND q = g^{(3)}g^{(4)},\]
which are integral binary quadratic forms by definition. We then deduce that
\[(p_{M_f},q_{M_f}) = (\lambda\cdot q, \lambda^{-1}\cdot p) \OR
(p_{M_f},q_{M_f}) = (\lambda\cdot p, \lambda^{-1}\cdot q)\]
for some $\lambda\in\bQ^\times$. In the former case, it is clear that $F$ is of type $1$. In the latter case, we have $\lambda=-1$ by Lemma~\ref{quadratic factor} and the fact that $\Delta(F)\neq0$, so $F$ is of type $2$.
\end{proof}

\section{Parametrizing forms in $V_{\bR,f}$ of non-zero discriminant}

Throughout this section, let $f(x,y) = \alpha x^2 + \beta xy + \gamma y^2$ be a real binary quadratic form with $\Delta(f)\neq0$ and $\alpha>0$. We shall give an alternative parametrization of $V_{\bR,f}^0$, different from (\ref{abc family}) and (\ref{abc family 2}), in terms of the regions
\begin{align}\label{Omega def}
\Omega^0 & = \{(L,K)\in\bR^2 \mid L^2 + 4K\neq0 \AND 2L^2 - K\neq0\},\\\notag
\Omega^+ & = \{(L,K) \in \bR^2\mid L^2 + 4K>0 \AND 2L^2 - K \neq0\},\\\notag
\Omega^- & = \{(L,K) \in \bR^2\mid L^2 + 4K<0 \AND 2L^2 - K>0\},
\end{align}
corresponding to the $L_f$- and $K_f$-invariants, as well as a parameter $t\in\bR$ arising from the \emph{orthogonal group of $f$}, defined by
\[O_f(\bR) = \{T\in\GL_2(\bR): \det(T) = \pm1 \AND f_T = \pm f\}.\]
Note that by (\ref{Delta LK}), for any $F\in V_{\bR,f}^0$, we have
\begin{align*}
(L_f(F),K_f(F))\in \Omega^+ &\iff \Delta(F)>0,\\
(L_f(F),K_f(F))\in \Omega^- &\iff \Delta(F)<0.
\end{align*}
First, we shall show that it suffices to consider $x^2+ y^2$ and $x^2-y^2$. It shall be helpful to recall (\ref{V bijection}) as well as the isomorphisms $\Theta_1$ and $\Theta_2$ defined in Subsection~\ref{notation sec}.

\begin{lemma}\label{Psi lemma}Define a matrix
\[T_f=\begin{pmatrix} \delta_f^{-1/4} & 0 \\ 0 & \delta_f^{1/4} \end{pmatrix}\cdot \frac{1}{ 2\sqrt{ \alpha}} \begin{pmatrix} 2\alpha & \beta \\ 0 & 2\end{pmatrix},\mbox{ where }\delta_f = \frac{|\Delta(f)|}{4}\]
Then, we have a well-defined bijective linear map
\[\begin{cases}\Psi_f:V_{\bR,x^2+y^2}\longrightarrow V_{\bR,f};\hspace{1em}\Psi_f(F) = F_{T_f}&\mbox{if $f$ is positive definite},\\\Psi_f:V_{\bR,x^2-y^2}\longrightarrow V_{\bR,f};\hspace{1em}\Psi_f(F) = F_{T_f}&\mbox{if $f$ is indefinite},\end{cases}\]
and we have $\det(\Psi_f) = 8\alpha^3|\Delta(f)|^{-3/2}$.
\end{lemma}
\begin{proof}The first claim holds by (\ref{V bijection}) and the fact 
\[\delta_f^{-1/2}\cdot f= \begin{cases} (x^2+y^2)_{T_f} &\mbox{if $f$ is positive definite},\\(x^2 - y^2)_{T_f}&\mbox{if $f$ is indefinite}.\end{cases}\]
Identifying $V_{\bR,x^2\pm y^2}$ and $V_{\bR,f}$ with $\bR^3$ via $\Theta_1$, we see from (\ref{abc family}) that
\begin{equation} \label{general parameter}  \Psi_f: \begin{pmatrix}a_4\\a_3\\a_2\end{pmatrix}\mapsto\begin{pmatrix} \frac{\alpha^2}{\delta_f} && 0 && 0 \\[0.5ex]
\frac{2\alpha\beta}{\delta_f} && \frac{\alpha}{\sqrt{\delta_f}} && 0\\[0.5ex]
\frac{3\beta^2}{2\delta_f}&& \frac{3\beta}{2\sqrt{\delta_f}} && 1
\end{pmatrix}\begin{pmatrix}a_4\\a_3\\a_2\end{pmatrix},\end{equation}
from which the second claim follows.
\end{proof}

In the subsequent subsections, we shall prove the following propositions.

\begin{proposition}\label{Phi+}There exists an explicit bijection
\[ \Phi:\Omega^+ \times [-\pi/4,\pi/4) \longrightarrow V_{\bR,{x^2+y^2}}^0,\]
defined as in (\ref{Phi+ def}), such that
\begin{enumerate}[(a)]
\item we have $L_{x^2+y^2}(\Phi(L,K,t)) = L$ and $K_{x^2+y^2}(\Phi(L,K,t)) = K$,
\item the Jacobian matrix of $\Theta_1\circ\Phi$ has determinant $-1/18$.
\end{enumerate}
\end{proposition}

\begin{proposition}\label{Phi-}There exist explicit injections
\[\Phi^{(1)},\Phi^{(2)} :\Omega^+\times \bR\longrightarrow V_{\bR,x^2-y^2}^0\AND\Phi^{(3)},\Phi^{(4)} :\Omega^-\times \bR\longrightarrow V_{\bR,x^2-y^2}^0,\]
defined as in (\ref{Phi- def}), with
\[V_{\bR,x^2-y^2}^0 = \Phi^{(1)}(\Omega^+ \times\bR) \sqcup \Phi^{(2)}(\Omega^+ \times\bR) \sqcup \Phi^{(3)}(\Omega^- \times\bR) \sqcup \Phi^{(4)}(\Omega^- \times\bR)\]
such that 
\begin{enumerate}[(a)]
\item we have $L_{x^2-y^2}(\Phi^{(i)}(L,K,t)) = L$ and $K_{x^2-y^2}(\Phi^{(i)}(L,K,t)) = K$,
\item the Jacobian matrix of $\Theta_1\circ\Phi^{(i)}$ has determinant $-1/18$,
\end{enumerate}
for all $i=1,2,3,4$.
\end{proposition}

In view of (\ref{reducible f shape}), we shall give another parametrization of $V_{\bR,f}$ when $\gamma=0$, which does not require reducing to the form $x^2 - y^2$ via Lemma~\ref{Psi lemma}.

\begin{proposition}\label{Phi f}Suppose that $\gamma=0$. Then, there exist explicit injections
\[\Phi_f^{(1)},\Phi_f^{(2)}:\Omega^0 \times\bR\longrightarrow V_{\bR,f}^0,\]
defined as in (\ref{Phi f def}), with
\[V_{\bR,f}^0 = \Phi_f^{(1)}(\Omega^0\times\bR)\sqcup \Phi_f^{(2)}(\Omega^0\times\bR)\]
such that
\begin{enumerate}[(a)]
\item we have $L_{f}(\Phi^{(i)}(L,K,t)) = L$ and $K_{f}(\Phi^{(i)}(L,K,t)) = K$,
\item the Jacobian matrix of $\Theta_2\circ\Phi_f^{(i)}$ has determinant $-1/18$,
\end{enumerate}
for both $i=1,2$.
\end{proposition}

For $t\in \bR$, we shall use the notation
\begin{equation}\label{T def}T^+(t) = \begin{pmatrix} \cos t & \sin t \\ - \sin t & \cos t \end{pmatrix}\AND  T^-(t) = \begin{pmatrix}\cosh t & \sinh t\\\sinh t & \cosh t\end{pmatrix},\end{equation}
which is an element of $O_{x^2+y^2}(\bR)$ and $O_{x^2-y^2}(\bR)$, respectively.

\subsection{Positive definite case}\label{para pos def}


Define
\begin{equation}\label{Phi+ def}\Phi:\Omega^+ \times [-\pi/4,\pi/4) \longrightarrow V_{\bR,{x^2+y^2}}^0;\hspace{1em}\Phi(L,K,t) = (F_{(L,K)})_{T^+(t)},\end{equation}
where 
\[F_{(L,K)}(x,y) = \frac{-3L + \sqrt{L^2 + 4K}}{24}x^4 + \frac{-L - \sqrt{L^2 + 4K}}{4} x^2 y^2 + \frac{-3L + \sqrt{L^2 + 4K}}{24} y^4.\]
The image of $\Phi$ lies in $V_{\bR,{x^2+y^2}}$ by (\ref{abc family}) and (\ref{V bijection}). Using Propositions~\ref{LK invariant} and~\ref{explicit LK} (a), it is easy to check that Proposition~\ref{Phi+} (a) holds.\\

Now, by (\ref{abc family}), an arbitrary $F\in V_{\bR,x^2+y^2}^0$ has the shape
\[ F(x,y) = a_4x^4 + a_3x^3y + a_2x^2y^2 - a_3xy^3 + a_4y^4.\]
Write $L = L_{x^2+y^2}(F)$ and $K = K_{x^2+y^2}(F)$. Note that $(L,K)\in\Omega^+$ because $\Delta(F)>0$ by (\ref{Delta LK}). For $t\in\bR$, a direct computation yields
\[ F_{T^+(t)}(x,y) = A(t)x^4 + B(t)x^3y + C(t)x^2y^2 - B(t)xy^3 + A(t)y^4,\]
where 
\[\begin{cases}A(t) = \dfrac{6a_4+a_2}{8} + \dfrac{2a_4-a_2}{8}\cos(4t) - \dfrac{a_3}{4}\sin(4t),\vspace{2mm}\\
B(t) = a_3\cos(4t) + \dfrac{2a_4-a_2}{2}\sin(4t),\vspace{2mm}\\
C(t) = \dfrac{6a_4+a_2}{4} - \dfrac{3(2a_4-a_2)}{4}\cos(4t) + \dfrac{3a_3}{2}\sin(4t).
\end{cases} \]
It is not hard to show that there exists a unique $t_0\in(-\pi/4,\pi/4]$ such that $B(t_0)=0$ and $2A(t_0)-C(t_0)>0$. Put $(A,C) = (A(t_0),C(t_0))$. Then, we have
\[(L,K) = (L_{x^2+ y^2}(F_{T^+(t_0)}),K_{x^2+y^2}(F_{T^+(t_0)}))  = (-6A - C, -2C(6A- C))\]
by Propositions~\ref{LK invariant} and~\ref{explicit LK} (a). We solve that $F_{T^+(t_0)}=F_{(L,K)}$, or equivalently 
\[F = (F_{(L,K)})_{T^+(-t_0)} = \Phi(L,K,-t_0).\]
Since $-t_0\in[-\pi/4,\pi/4)$ is uniquely determined by $F$, this shows that $\Phi$ is a bijection.\\

Finally, the above calculation also yields
\[ (\Theta_1\circ\Phi)(L,K,t) =(\Phi_{1}(L,K,t),\Phi_2(L,K,t),\Phi_3(L,K,t)),\]
where
\begin{equation} \label{pos def para}
\begin{cases} \Phi_{1}(L,K,t) = -\dfrac{L}{8} + \dfrac{\sqrt{L^2 + 4K}}{24} \cos(4t),\vspace{2mm}\\
\Phi_{2}(L,K,t) = \dfrac{\sqrt{L^2 + 4K}}{6} \sin(4t),\vspace{2mm}\\
\Phi_{3}(L,K,t) = -\dfrac{L}{4} - \dfrac{\sqrt{L^2 + 4K}}{4} \cos(4t).
\end{cases}\end{equation}
By a direct computation, we then see that Proposition~\ref{Phi+} (b) holds.

\subsection{Indefinite case}
\label{geometric para indefinite}


Define
\begin{equation}\label{Phi- def}\begin{cases} \Phi^{(i)}: \Omega^+ \times \bR \longrightarrow V_{\bR,x^2-y^2}^0; \hspace{1em}\Phi^{(i)}(L,K,t) = (F_{(L,K)}^{(i)})_{T^-(t)} & \mbox{for $i=1,2$},\\
 \Phi^{(i)}: \Omega^- \times \bR \longrightarrow V_{\bR,x^2-y^2}^0; \hspace{1em}\Phi^{(i)}(L,K,t) = (F_{(L,K)}^{(i)})_{T^-(t)} & \mbox{for $i=3,4$},
\end{cases}\end{equation}
where
\begin{align*}
F_{(L,K)}^{(i)}(x,y)&=\frac{3L + (-1)^{i}\sqrt{L^2 + 4K}}{24}x^4 + \frac{-L +(-1)^{i}\sqrt{L^2 + 4K}}{4} x^2 y^2 \\&\hspace{7.5cm}+\frac{3L + (-1)^{i}\sqrt{L^2 + 4K}}{24}y^4\end{align*}
for $i=1,2$, and 
\[F_{(L,K)}^{(i)}(x,y)=\frac{(-1)^{i}\sqrt{2L^2-K}}{3}x^3y-Lx^2y^2+\frac{(-1)^{i}\sqrt{2L^2-K}}{3}xy^3\]
for $i=3,4$. The images of $\Phi^{(1)},\Phi^{(2)},\Phi^{(3)},\Phi^{(4)}$ lie in $V_{\bR,{x^2-y^2}}$ by (\ref{abc family}) and (\ref{V bijection}). Using Propositions~\ref{LK invariant} and~\ref{explicit LK} (a), it is easy to check that Proposition~\ref{Phi-} (a) holds.\\

Now, by (\ref{abc family}), an arbitrary $F\in V_{\bR,x^2-y^2}^0$ has the shape
\[F(x,y) = a_4x^4 + a_3x^3y + a_2x^2y^2 + a_3xy^3 + a_4y^4.\]
Write $L = L_{x^2-y^2}(F)$ and $K = K_{x^2-y^2}(F)$. For $t\in\bR$, a direct computation yields
\[ F_{T^-(t)}(x,y) = A(t)x^4 + B(t)x^3y + C(t)x^2y^2 + B(t)xy^3 + A(t)y^4,\]
where
\[\begin{cases}
A(t) = \dfrac{6a_4-a_2}{8} + \dfrac{2a_4+a_2}{8}\cosh(4t) + \dfrac{a_3}{4}\sinh(4t),\vspace{2mm}\\
B(t) = a_3\cosh(4t) + \dfrac{2a_4+a_2}{2}\sinh(4t),\vspace{2mm}\\
C(t) = -\dfrac{6a_4-a_2}{4} + \dfrac{3(2a_4+a_2)}{4}\cosh(4t) + \dfrac{3a_3}{2} \sinh(4t).
\end{cases}\]
Note that $\frac{d}{dt}A(t) = \frac{1}{2}B(t)$. It is not hard to check that:
\begin{itemize}
\item If $\Delta(F)>0$, then there is a unique $t_0\in\bR$ such that $B(t_0)=0$.
\item If $\Delta(F)<0$, then $B(t)\neq0$ for all $t\in\bR$, and there is a unique $t_0\in\bR$ such that $A(t_0)=0$.
\end{itemize}
Put $(A,B,C) = (A(t_0),B(t_0),C(t_0))$. Then, we have
\begin{align*}(L,K) & =  (L_{x^2-y^2}(F_{T^-(t_0)}), K_{x^2-y^2}(F_{T^-(t_0)})) =  \begin{cases} (6A - C, 2C(6A+C)) & \mbox{if $\Delta(F)>0$},\\
(-C, -9B^2 + 2C^2)&\mbox{if $\Delta(F)<0$}.\end{cases}\end{align*}
by Propositions~\ref{LK invariant} and~\ref{explicit LK} (a). We solve that $F_{T^-(t_0)} = F_{(L,K)}^{(i)}$, or equivalently
\[F = (F_{(L,K)}^{(i)})_{T^-(-t_0)} = \Phi^{(i)}(L,K,-t_0),\mbox{ for exactly one $i\in\{1,2,3,4\}$}.\]
Since $t_0$ is uniquely determined by $F$, this shows that $\Phi^{(1)},\Phi^{(2)},\Phi^{(3)},\Phi^{(4)}$ are all injections, and that the stated disjoint union holds.\\

Finally, the above calculation also yields
\[ (\Theta_1\circ\Phi^{(i)})(L,K,t) = (\Phi_1^{(i)}(L,K,t), \Phi_2^{(i)}(L,K,t), \Phi_3^{(i)}(L,K,t)),\]
where
\begin{equation}\label{indef para 1}\begin{cases}
\Phi_1^{(i)}(L,K,t) = \dfrac{L}{8} + \dfrac{(-1)^i\sqrt{L^2+4K}}{24}\cosh(4t),\vspace{2mm} \\
\Phi_2^{(i)}(L,K,t) = \dfrac{(-1)^i\sqrt{L^2+4K}}{6}\sinh(4t),\vspace{2mm} \\
\Phi_3^{(i)}(L,K,t) = -\dfrac{L}{4} + \dfrac{(-1)^i\sqrt{L^2+4K}}{4}\cosh(4t),\\
\end{cases} \end{equation}
for $i=1,2$, and 
\begin{equation}\label{indef para 2}\begin{cases}
\Phi_1^{(i)}(L,K,t) = \dfrac{L}{8} - \dfrac{L}{8}\cosh(4t) + \dfrac{(-1)^i\sqrt{2L^2-K}}{12}\sinh(4t),\vspace{2mm} \\
\Phi_2^{(i)}(L,K,t) = \dfrac{(-1)^i\sqrt{2L^2-K}}{3}\cosh(4t) - \dfrac{L}{2}\sinh(4t),\vspace{2mm} \\
\Phi_3^{(i)}(L,K,t) = -\dfrac{L}{4} - \dfrac{3L}{4}\cosh(4t) + \dfrac{(-1)^i\sqrt{2L^2-K}}{2}\sinh(4t),\\
\end{cases} \end{equation}
for $i=3,4$. By a direct computation, we then see that Proposition~\ref{Phi-} (b) holds.

\subsection{Reducible case}\label{red para section}

Suppose $\gamma=0$. For $t\in\bR$, put
\[ T(t) =  \begin{pmatrix} e^{-t} &  0 \\[0.5ex] \dfrac{2\alpha \sinh t}{\beta} &  e^{t} \end{pmatrix},\]
which is an element of $O_f(\bR)$. Define
\begin{equation}\label{Phi f def}\Phi_f^{(i)}:\Omega^0 \times\bR\longrightarrow V_{\bR,f}^0 ; \hspace{1em}\Phi_f^{(i)}(L,K,t) = (F_{f,(L,K)}^{(i)})_{T(t)}\hspace{1em}\mbox{for $i=1,2$},\end{equation}
where
\begin{align*}F_{f,(L,K)}^{(i)}(x,y)&=\left(\frac{L^2+(-1)^i72\alpha^2L+4K+144\alpha^4}{(-1)^i144\beta^2}\right) x^4 + \left(\frac{\alpha L + (-1)^i 4 \alpha^3}{\beta}\right) x^3 y\\&\hspace{2.5cm}+ \left(\frac{L +(-1)^i 12\alpha^2}{2}\right)x^2 y^2 +(-1)^i 4 \alpha \beta xy^3 + (-1)^i\beta^2 y^4.\end{align*}
The images of $\Phi_f^{(1)},\Phi_f^{(2)}$ lie in $V_{\bR,f}$ by (\ref{abc family 2}) and (\ref{V bijection}). Using Propositions~\ref{LK invariant} and~\ref{explicit LK} (b), it is easy to check that Proposition~\ref{Phi f} (a) holds.\\

Now, by (\ref{abc family 2}), an arbitrary $F\in V_{\bR,f}^0$ has the shape
\begin{equation}\label{reducible generic}F(x,y) = a_4x^4 +  \left(\frac{2\alpha(\beta^2 a_2 - 4 \alpha^2a_0)}{\beta^3}\right)x^3 y + a_2x^2 y^2 +\left(\frac{4 \alpha a_0}{\beta}\right) xy^3 + a_0y^4.\end{equation}
Write $L = L_f(F)$ and $K= K_f(F)$. For $t\in\bR$, a direct computation yields
\[ F_{T(t)}(x,y) = A(t)x^4 + (*)x^3y + B(t)x^2y^2 + (*)xy^3 + C(t)y^4,\]
where
\[\label{red ABC} \begin{cases}
A(t) = e^{-4t}a_4 + \dfrac{\alpha^2}{\beta^2}(e^{4t}-1)e^{-4t}a_2 + \dfrac{\alpha^4}{\beta^4}(e^{4t}-1)(e^{4t}-5)e^{-4t}a_0\vspace{2mm},\\
B(t) = a_2 + \dfrac{6\alpha^2}{\beta^2}(e^{4t}-1)a_0,\vspace{2mm}\\
C(t) = e^{4t}a_0.
\end{cases}\]
Since $\Delta(F)\neq0$, we have $(-1)^i a_0 > 0$ for a unique $i\in\{1,2\}$, and there is a unique $t_{0}\in\bR$ such that $C(t_{0}) = (-1)^i\beta^2$. Put $(A,B) = (A(t_0),B(t_0))$. Then, we have
\begin{align*}(L,K)  &= (L_f(F_{T(t_0)}) , K_f(F_{T(t_0)})) \\&= (2B - (-1)^i12\alpha^2, -B^2 + (-1)^i36\beta^2A - (-1)^i24\alpha^2B + 144\alpha^4),\end{align*}
by Propositions~\ref{LK invariant} and~\ref{explicit LK} (b). We solve that $F_{T(t_0)} = F_{f,(L,K)}^{(i)}$, or equivalently 
\[F =  (F_{f,(L,K)}^{(i)})_{T(-t_0)} = \Phi_f^{(i)}(L,K,-t_0).\]
Since $t_0$ and $i$ are uniquely determined by $F$, this shows that $\Phi_f^{(1)}$ and $\Phi_f^{(2)}$ are both injections,  and that the stated disjoint union holds.\\

Finally, the above calculation also yields
\[ (\Theta_2\circ\Phi_f^{(i)})(L,K,t) = (\Phi^{(i)}_{f,1}(L,K,t),\Phi_{f,2}^{(i)}(L,K,t),\Phi^{(i)}_{f,3}(L,K,t)),\]
where
\begin{equation} \label{red para}
\begin{cases} \Phi^{(i)}_{f,1}(L,K,t)=\dfrac{(-1)^i e^{-4t}}{144\beta^2}(L^2+4K) + \dfrac{\alpha^2}{2\beta^2}L
+ \dfrac{(-1)^i\alpha^4e^{4t}}{\beta^2}, \vspace{2mm}\\
\Phi_{f,2}^{(i)}(L,K,t)=\dfrac{L}{2}+(-1)^i6\alpha^2 e^{4t},\vspace{2mm}\\
\Phi^{(i)}_{f,3}(L,K,t)=(-1)^i\beta^2 e^{4t}.
\end{cases}\end{equation}
By a direct computation, we then see that Proposition~\ref{Phi f} (b) holds.

\section{Definition of a bounded semi-algebraic set}
\label{proof sec}

Throughout this section, let $f(x,y) = \alpha x^2 + \beta xy + \gamma y^2$ be an integral and primitive binary quadratic form with $\Delta(f)\neq0$ and $\alpha>0$, in the shape (\ref{reducible f shape}) whenever $f$ is reducible. As we have already explained in Subsection~\ref{notation sec}, the proof of Theorem~\ref{Small Gal MT} is reduced to counting points in the lattices in (\ref{Lambda def}), which in turn amounts to certain volume computations, by the result below.

\begin{proposition}[Davenport's lemma] \label{Davenport}Let $\R$ be a bounded semi-algebraic multi-set in $\bR^n$ having maximum multiplicity $m$ and which is defined by at most $k$ polynomial inequalities, each having degree at most $\ell$. Then, the number of integral lattice points (counted with multiplicity) contained in the region $\R$ is 
\[\Vol(\R) + O(\max\{\Vol(\ol{\R}), 1\}),\]
where $\Vol(\ol{\R})$ denotes the greatest $d$-dimensional volume of any projection of $\R$ onto a coordinate subspace by equating $n-d$ coordinates to zero, with $1 \leq d \leq n-1$. The implied constant in the second summand depends only on $n,m,k,\ell$. 
\end{proposition}
\begin{proof}This is a result of Davenport \cite{Dav}, and the above formulation is due to Bhargava and Shankar in \cite[Proposition 2.6]{BhaSha}.
\end{proof}

For $X>0$, define
\[ V_{\bR,f}^0(X) = \{F\in V_{\bR,f}^0 : H_f(F)\leq X\}\AND V_{\bZ,f}^0(X) = \{F\in V_{\bZ,f}^0 : H_f(F)\leq X\}.\]
However, to prove Theorem~\ref{Small Gal MT}, we cannot apply Proposition~\ref{Davenport} directly to
\[\Theta_{w(f)}(V_{\bR,f}^0(\bR)),\mbox{ where }w(f) = \begin{cases}
1 & \mbox{if $f$ is irreducible},\\ 2 & \mbox{if $f$ is reducible},
\end{cases}\]
as in Subsection~\ref{notation sec}, to count the lattice points in $\Theta_{w(f)}(V_{\bZ,f})\subset\Lambda_{f,w(f)}$ because
\begin{enumerate}[(1)] 
\item the set $\Theta_{w(f)}(V_{\bR,f}^0(X))$ is unbounded when $f$ is indefinite,
\item distinct forms in $V_{\bZ,f}^0(X)$ might be $\GL_2(\bZ)$-equivalent.
\end{enumerate}
Recall (\ref{Omega def}) and define
\[\Omega^*(X) = \{(L,K)\in\Omega^*: \max\{L^2,|K|\}\leq X\}\mbox{ for }*\in\{0,+,-\}.\]
In the notation of Lemma~\ref{Psi lemma} as well as Propositions~\ref{Phi+},~\ref{Phi-}, and~\ref{Phi f}, we have
\begin{equation}\label{para} V_{\bR,f}^0(X) = \begin{cases}
(\Psi_f\circ\Phi)(\Omega^+(X)\times [-\pi/4,\pi/4)),\\
\bigsqcup\limits_{i=1}^{2}(\Psi_f\circ\Phi^{(i)})(\Omega^+(X)\times \bR)\sqcup\bigsqcup\limits_{i=3}^{4}(\Psi_f\circ\Phi^{(i)})(\Omega^-(X)\times \bR),\\
\bigsqcup\limits_{i=1}^{2}\Phi_f^{(i)}(\Omega^0(X)\times\bR),
\end{cases}\end{equation}
respectively, if $f$ is positive definite, indefinite, and reducible. We shall overcome the two issues above by restricting the values for $t\in\bR$.\\

For brevity, in this section, write
\[D_f = |\Delta(f)|\AND\delta_f = D_f/4,\]
as in Theorem~\ref{Small Gal MT} and Lemma~\ref{Psi lemma}, respectively.

\begin{definition}\label{S(X) def}
If $f$ is positive definite, define
\[ \S_f(X) = (\Psi_f\circ\Phi)(\Omega^+(X) \times [-\pi/4,\pi/4)). \]
If $f$ is reducible, define
\[ \S_f(X) = \bigsqcup_{i=1}^{2}\Phi_f^{(i)}(\Omega^0(X)\times [t_{f,1},t_{f,2}]) \mbox{ for }t_{f,1} = -\frac{\log 8}{4} \AND t_{f,2} =\frac{\log(5X/18)}{4}.\]
If $f$ is indefinite and irreducible, define
\[\S_f(X) =   \bigsqcup_{i=1}^2(\Psi_f\circ\Phi^{(i)})(\Omega^+(X)\times [0,t_{D_f})) \sqcup \bigsqcup_{i=3}^4(\Psi_f\circ\Phi^{(i)})(\Omega^-(X)\times [0,t_{D_f})),\]
where $t_{D_f}$ is defined as in Theorem~\ref{Small Gal MT} (c).
\end{definition}

The goal of this section to prove the following preliminary results and estimates:

\begin{proposition}\label{issue1 prop} The set $\Theta_{w(f)}(S_f(X))$ is bounded, semi-algebraic, and definable by an absolutely bounded number of polynomial inequalities whose degrees are absolutely bounded.
\end{proposition}

\begin{proposition}\label{issue2 prop}The following statements hold.
\begin{enumerate}[(a)]
\item A form in $V_{\bZ,f}^0(X)$ is $\GL_2(\bZ)$-equivalent to at least one form in $\S_f(X)$.
\item A form in $V_{\bZ,f}^0(X)$ for which $\Delta(F)\neq\square$ is $\GL_2(\bZ)$-equivalent to exactly $r_f$ forms in $\S_f(X)$, where $r_f$ is defined as in Theorem~\ref{Small Gal MT}.
\end{enumerate}
\end{proposition}

\subsection{Alternative description} First, we shall give an alternative description of the set $\S_f(X)$ in terms of the coefficients of the forms in $V_{\bR,f}^0(X)$.

\begin{lemma}\label{Sf positive definite}If $f$ is positive definite, then $\S_f(X) = V_{\bR,f}^0(X)$.
\end{lemma}
\begin{proof}This is clear from (\ref{para}). \end{proof}

\begin{lemma}\label{Sf reducible} If $f$ is reducible, then
\[\S_f(X) = \{F\in V_{\bR,f}^0(X) : \beta^2/8 \leq |C_F|\leq 5\beta^2X/18\},\]
where $C_F$ denotes the $y^4$-coefficient of $F$.
\end{lemma}
\begin{proof}For $i=1,2$ and for any $F=\Phi_f^{(i)}(L,K,t)$, we have $C_{F} = (-1)^i\beta^2 e^{4t}$ by (\ref{red para}), and the claim is then clear from (\ref{para}).
\end{proof}

\begin{lemma}\label{Sf indefinite}If $f$ is an indefinite and irreducible, then
\[\S_f(X) = \{F\in V_{\bR,f}^0(X): 1\leq E_{f,1}(F)Z_f(F)/E_{f,2}(F)< e^{8t_{D_f}}\},\]
where in the notation of Proposition~\ref{explicit LK} (a), we define
\[ E_{f,1}(F) = L_{f,1}(F) - \sqrt{D_f}L_{f,2}(F) \AND
E_{f,2}(F) = L_{f,1}(F) + \sqrt{D_f}L_{f,2}(F),\]
and for $F$ in the image of $\Psi_f\circ\Phi^{(i)}$, we define
\[ Z_f(F) = \begin{cases} 1 &\mbox{for $i=1,2$},\\ \dfrac{L_f(F)^2 + 4K_f(F)}{(4L_f(F) - (-1)^i2\sqrt{2L_f(F)^2 - K_f(F)})^2} &\mbox{for $i=3,4$}.\end{cases}\]
\end{lemma}
\begin{proof}For $i=1,2,3,4$, consider $F = (\Psi_f\circ\Phi^{(i)})(L,K,t)$. For $k=1,2$, we have
\[ E_{f,k}(F) = \begin{cases} (-1)^i 2\alpha^2\sqrt{L_f(F)^2+4K_f(F)}e^{(-1)^{k+1}4t}/3 & \mbox{if $i=1,2$},\\  -2\alpha^2(3L_f(F) + (-1)^{k+i}2\sqrt{2L_f(F)^2-K_f(F)})e^{(-1)^{k+1}4t}/3 &\mbox{if $i=3,4$},\end{cases}\]
by a direct computation using (\ref{general parameter}), (\ref{indef para 1}), and (\ref{indef para 2}). We then see that
\[ E_{f,1}(F)Z_f(F)/E_{f,2}(F) = e^{8t},\]
from which the claim follows.
\end{proof}

\subsection{Proof of Proposition~\ref{issue1 prop}} 

From (\ref{pos def para}), (\ref{indef para 1}), (\ref{indef para 2}), and (\ref{red para}), it is clear that the set $\S_f(X)$ is bounded. Thus, it remains to show that $\S_f(X)$ is a semi-algebraic set definable by an absolutely bounded number of polynomial inequalities whose degrees are absolutely bounded.

\subsubsection{The case when $f$ is positive definite or reducible} The claim follows immediately from Lemmas~\ref{Sf positive definite} and~\ref{Sf reducible} as well as Proposition~\ref{explicit LK}. 

\subsubsection{The case when $f$ is indefinite and irreducible} 

The only problem is that $Z_f(F)$ is not a polynomial in the $x^4$, $x^3y$, and $x^2y^2$-coefficients of $F$. We shall resolve this issue in Lemma~\ref{poly ineq} below. The claim then follows from Lemma~\ref{Sf indefinite} and Proposition~\ref{explicit LK}.

\begin{lemma}\label{poly ineq}For $i=3,4$, let $F\in(\Psi_f\circ\Phi^{(i)})(\Omega^-\times\bR)$. Then, the condition
\[1\leq E_{f,1}(F)Z_f(F)/E_{f,2}(F) < e^{8t_{D_f}}\]
is equivalent to an absolutely bounded number of polynomial inequalities in the variables $L_f(F),K_f(F),E_{f,1}(F),E_{f,2}(F)$ whose degrees are absolutely bounded.
\end{lemma}
\begin{proof}For brevity, define
\begin{align*} Y_{f,1}(F) &= - E_{f,1}(F)(L_f(F)^2 + 4K(F)) +  E_{f,2}(F)(17L_f(F)^2 - 4K_f(F)),\\
Y_{f,2}(F) & = - E_{f,1}(F)(L_f(F)^2 + 4K_f(F)) + e^{8t_{D_f}}E_{f,2}(F)(17L_f(F)^2 - 4K_f(F)),\end{align*}
as well as write
\[ (L,K, E_1,E_2,Z, Y_1,Y_2) = (L_f(F), K_f(F), E_{f,1}(F), E_{f,2}(F),Z_f(F), Y_{f,1}(F), Y_{f,2}(F)).\]
Note that $L^2+4K<0$ by (\ref{Delta LK}) because $\Delta(F)<0$. This implies that $Z<0$ and so the stated condition may be rewritten as
\[\label{remove abs}\begin{cases}E_2 \leq E_1Z < e^{8t_{D_f}}E_2&\text{if $E_2 > 0$, which is equivalent to $i=3$},\\
E_2 \geq E_1Z > e^{8t_{D_f}}E_2&\text{if $E_2<0$, which is equivalent to $i=4$}.\end{cases}\]
By rearranging, we may further rewrite the above as
\[\label{xmp ineq}12E_2 L\sqrt{2L^2 - K}\leq (-1)^iY_1 \AND 12e^{8t_{D_f}}E_2L\sqrt{2L^2 - K} > (-1)^iY_2.\]
From here, we shall consider the different possibilities for the signs of $E_2$, $L$, $Y_1,Y_2$. For example, when $E_2 > 0$ and $L \geq 0$, the above is equivalent to $Y_1\leq 0$ and
\[\begin{cases}
(12E_2 L)^2(2L^2 - K) \leq Y_1^2&\text{if $Y_2 > 0$},\\
(12E_2 L)^2(2L^2 - K) \leq Y_1^2\AND (12e^{8t_{D_f}}E_2L)^2(2L^2 - K) > Y_2^2&\text{if $Y_2 \leq 0$}.
\end{cases}\]
The other cases are analogous. We then see that the claim holds.
\end{proof}


\subsection{Integral orthogonal groups}We shall require an explicit description of 
\[ O_f(\bZ) = O_f(\bR)\cap \GL_2(\bZ).\]
In the notation of Lemma~\ref{Psi lemma}, observe that 
\begin{equation}\label{O relation}
O_f(\bR) = \begin{cases} T_f^{-1}(O_{x^2+y^2}(\bR))T_f &\mbox{if $f$ is positive definite},\\ T_f^{-1}(O_{x^2-y^2}(\bR)) T_f &\mbox{if $f$ is indefinite}.\end{cases}
\end{equation}
Moreover, it is well-known that
\begin{align*}\label{O principal} O_{x^2+y^2}(\bR) &= \{J_k T^+(t) : k\in\{1,4\}\AND t\in\bR\},\\\notag
O_{x^2-y^2}(\bR) & = \{\pm J_kT^-(t) : k\in\{1,2,3,4\} \AND t\in\bR\},\end{align*}
where $T^+(t)$ and $T^-(t)$ are defined as in (\ref{T def}), and 
\begin{equation}\label{J def}J_1 = \begin{pmatrix} 1&0 \\ 0&1\end{pmatrix},\, J_2 = \begin{pmatrix} 0&1 \\ 1&0\end{pmatrix},\,J_3 = \begin{pmatrix} 0&1 \\ -1&0\end{pmatrix},\, J_4 = \begin{pmatrix} 1&0 \\ 0&-1\end{pmatrix}.\end{equation}
We shall need the following lemma.

\begin{lemma} \label{Dickson ambiguous}Suppose that $T\in O_f(\bZ)\setminus\{\pm I_{2\times 2}\}$ has finite order. Then, the form $f$ is $\GL_2(\bZ)$-equivalent to a form of the shape
\[\begin{cases}
 x^2 + y^2,\, x^2 + xy + y^2, \OR  ax^2 + bxy - ay^2 & \mbox{if $\det(T) = 1$},\\
xy,\, x^2-y^2,\,ax^2+cy^2,\OR ax^2 + bxy + ay^2&\mbox{if $\det(T)=-1$},
\end{cases}\]
for some integers $a,b$, and $c$.
\end{lemma}
\begin{proof}By \cite[Chapter IX]{Newman}, for example, a finite cyclic subgroup of $\GL_2(\bZ)$ not contained in $\{\pm I_{2\times2}\}$ is conjugate to the subgroup generated by one of the following:
\[ \begin{pmatrix} 0 & 1 \\ -1 & -1 \end{pmatrix},\,
\begin{pmatrix}0 & 1 \\ -1 & 0 \end{pmatrix},\,
\begin{pmatrix}0 & -1 \\ 1 & 1 \end{pmatrix},\,
\begin{pmatrix} 1  & 0 \\ 0 & -1 \end{pmatrix},\,
\begin{pmatrix} 0 & 1 \\ 1 & 0\end{pmatrix}.\]
We then deduce that there exists $P\in\GL_2(\bZ)$ such that $Q = P^{-1}TP$ is equal to one of the following matrices up to sign:
\[ \begin{pmatrix} 0&1 \\ -1& -1\end{pmatrix},\,
\begin{pmatrix} -1&-1 \\ 1& 0\end{pmatrix},\,
 \begin{pmatrix} 0&1 \\ -1& 0\end{pmatrix},\,
 \begin{pmatrix} 1&0 \\ 0& -1\end{pmatrix},\,
 \begin{pmatrix} 0&1 \\ 1& 0\end{pmatrix}.\]
Since $f$ is primitive with $\alpha>0$ by assumption and $(f_P)_Q = \pm f_P$, we then check that $f_P$ must have one of the stated shapes.
\end{proof}

\begin{proposition}\label{Of pos def}Suppose that $f$ is positive definite. Then, we have
\[O_f(\bZ) = \{\pm I_{2\times2}\}\]
if $f$ is not $\GL_2(\bZ)$-equivalent to the forms below, and the group $O_f(\bZ)$ is equal to
\[\begin{cases}
\{\pm I_{2\times 2},\pm\left(\begin{smallmatrix}0&1\\-1&0\end{smallmatrix}\right),\pm\left(\begin{smallmatrix} 1& 0\\ 0 & -1 \end{smallmatrix}\right),\pm\left(\begin{smallmatrix} 0 & 1 \\ 1 & 0 \end{smallmatrix}\right)\}&\text{\hspace{-2.5mm}if $f(x,y) = x^2+y^2$},\\[0.5ex]
\{\pm I_{2\times2}\pm\left(\begin{smallmatrix}1&1\\-1&0\end{smallmatrix}\right),
\pm\left(\begin{smallmatrix}0&-1\\1&1\end{smallmatrix}\right),&\text{\hspace{-2.5mm}if $f(x,y) = x^2 + xy + y^2$},\\\hspace{1.25cm}
\pm\left(\begin{smallmatrix} 1& 1\\ 0 & -1 \end{smallmatrix}\right),\pm\left(\begin{smallmatrix} 0 & 1 \\ 1 & 0 \end{smallmatrix}\right),\pm\left(\begin{smallmatrix} -1& 0\\ 1 & 1 \end{smallmatrix}\right)\}&\\[0.5ex]
\{\pm I_{2\times2},\pm\left(\begin{smallmatrix} 1 & 0 \\ 0 & -1 \end{smallmatrix}\right)\}& \text{\hspace{-2.5mm}if $f(x,y) = \alpha x^2 + \gamma y^2$ for $\alpha \neq\gamma$},\\[0.5ex]
\left\{\pm I_{2 \times 2}, \pm \left(\begin{smallmatrix} 0 & 1 \\ 1 & 0 \end{smallmatrix} \right) \right\} & \text{\hspace{-2.5mm}if $f(x,y) = \alpha x^2 + \beta xy + \alpha y^2$ for $\beta\notin\{0,\alpha\}$}.
\end{cases}\]
\end{proposition}
\begin{proof}Elements in $O_f(\bZ)$ have finite order by (\ref{O relation}) and so the first claim follows from Lemma~\ref{Dickson ambiguous}. Using (\ref{O relation}), we compute that elements in $O_f(\bR)$ are of the forms
\[\begin{pmatrix}
\phi_t + \frac{\beta \psi_t}{2\sqrt{\delta_f}}& \frac{\gamma\psi_t}{\sqrt{\delta_f}}\vspace{1mm}\\
-\frac{\alpha\psi_t}{\sqrt{\delta_f}} & \phi_t - \frac{\beta\psi_t}{2\sqrt{\delta_f}}
\end{pmatrix}
\AND\begin{pmatrix}
\phi_t - \frac{\beta\psi_t}{2\sqrt{\delta_f}}&\frac{\beta}{\alpha}\left(\phi_t- \frac{\beta\psi_t}{2\sqrt{\delta_f}}\right) +\frac{\gamma\psi_t}{\sqrt{\delta_f}}\vspace{2mm}\\
\frac{\alpha\psi_t}{\sqrt{\delta_f}} & - \phi_t - \frac{\beta\psi_t}{2\sqrt{\delta_f}}
\end{pmatrix},\]
where $t\in\bR$ and $(\phi_t,\psi_t) = (\cos t,\sin t)$. With the help of the proof of Lemma~\ref{Dickson ambiguous}, it is not hard to check that $O_f(\bZ)$ is as claimed.
\end{proof}

\begin{proposition}\label{Of reducible}Suppose that $f$ is reducible. Then, the group $O_f(\bZ)$ is equal to
\[ \begin{cases}
\left \{ \pm I_{2\times2}\right \} &\text{if $\beta\nmid\alpha^2+1$ and $\beta\nmid \alpha^2-1$},\\[0.5ex]
\left\{\pm I_{2\times 2},\pm\left(\begin{smallmatrix}\alpha&&\beta\\-\frac{\alpha^2+1}{\beta}&&-\alpha\end{smallmatrix}\right)\right\}&\text{if $\beta\mid\alpha^2+1$ and $\beta\nmid \alpha^2-1$},\\[1ex]
\left\{\pm I_{2\times2},\pm\left(\begin{smallmatrix}\alpha&&\beta\\-\frac{\alpha^2-1}{\beta}&&-\alpha\end{smallmatrix}\right)\right\}&\text{if $\beta\nmid\alpha^2+1$ and $\beta\mid \alpha^2-1$},\\[1ex]
\left\{\pm I_{2\times2}, \pm \left(\begin{smallmatrix} -1 & 0 \\2 & 1 \end{smallmatrix}\right),\pm\left(\begin{smallmatrix} 1 & 1 \\ -2 & -1\end{smallmatrix}\right),\pm\left(\begin{smallmatrix}1 & 1 \\ 0 & -1\end{smallmatrix}\right)\right\}&\text{if }f(x,y) = x^2 + xy,\\[1ex]
\left\{ \pm I_{2\times 2}, \pm\left(\begin{smallmatrix} -1 & 0 \\ 1 & 1 \end{smallmatrix}\right),\pm\left(\begin{smallmatrix} 1 & 2 \\ -1 & -1 \end{smallmatrix}\right),\left(\begin{smallmatrix} 1 & 2 \\ 0 & -1 \end{smallmatrix}\right)\right\}&\text{if }f(x,y) = x^2 +  2xy.
\end{cases}\]
\end{proposition}
\begin{proof}Using (\ref{O relation}), we compute that elements in $O_f(\bR)$ are of the forms
\[\pm\begin{pmatrix} \phi_t - \psi_t  & 0 \vspace{2mm}\\ \frac{2\alpha\psi_t}{\beta} & \phi_t + \psi_t\end{pmatrix} 
\AND \pm\begin{pmatrix} \phi_t + \psi_t  & \frac{\beta}{\alpha}(\phi_t +\psi_t) \vspace{2mm}\\ -\frac{2\alpha\psi_t}{\beta} & -\phi_t- \psi_t\end{pmatrix},\]
where $t\in\bR$ and $(\phi_t,\psi_t)\in\{(\cosh t,\sinh t),(\sinh t,\cosh t)\}$. For the matrix on the left to have integer entries, necessarily 
\[2\cosh t,2\sinh t\in\bZ\mbox{ so }(2\cosh t, 2\sinh t) = (2,0).\]
Similarly, for the matrix on the right to have integer entries, necessarily 
\[2\alpha\cosh t,2\alpha\sinh t,(\cosh t +\sinh t)/\alpha\in \bZ \mbox{ so }(2\alpha\cosh t,2\alpha\sinh t) = (\alpha^2+1,\alpha^2-1).\]
We then deduce that
\[ O_f(\bZ) = \left\{\pm I_{2\times2},\pm\left(\begin{smallmatrix} -1&&0\\2\alpha/\beta&&1\end{smallmatrix}\right), \left(\begin{smallmatrix}\alpha&&\beta\\-(\alpha^2\pm1)/\beta&&-\alpha\end{smallmatrix}\right)\right\}\cap \GL_2(\bZ).\]
Since $f$ has the shape (\ref{reducible f shape}) by assumption, we have
\[ \beta\mid \alpha^2 + 1\AND \beta\mid \alpha^2-1\iff \alpha=1\AND \beta\in\{1,2\},\]
and we see that the claim indeed holds.
\end{proof}

\begin{proposition}\label{Of indefinite}Suppose that $f$ is indefinite and irreducible. Define
\[ G_f(\bZ) = \{\pm T_{D_f}^n : n\in\bZ\},\mbox{ where }T_{D_f}=\begin{pmatrix} \frac{u_{D_f} - \beta v_{D_f}}{2} & - \gamma v_{D_f} \vspace{2mm}\\ \alpha v_{D_f} & \frac{u_{D_f} + \beta v_{D_f}}{2} \end{pmatrix}\]
and $(u_{D_f},v_{D_f})\in\bN^2$ is the least solution to $x^2 - D_fy^2 = \pm4$. Then, we have
\[ O_f(\bZ) = G_f(\bZ)\]
if $f$ is not $\GL_2(\bZ)$-equivalence to the forms below, and the group $O_f(\bZ)$ is equal to
\[ \begin{cases}
G_f(\bZ)\sqcup G_f(\bZ)\left(\begin{smallmatrix}1&\beta/\alpha \\ 0 & -1 \end{smallmatrix}\right) &\mbox{if $f(x,y) = \alpha x^2 + \beta xy + \gamma y^2$ with $\alpha\mid \beta$},\\
G_f(\bZ)\sqcup G_f(\bZ)\left(\begin{smallmatrix}0&1\\-1&0 \end{smallmatrix}\right)&\mbox{if $f(x,y) = \alpha x^2 + \beta xy - \alpha y^2$}.
\end{cases}\]
\end{proposition}
\begin{proof}
By (\ref{O relation}), elements in $O_f(\bR)$ of infinite order are of the shape
\[\pm \begin{pmatrix} \phi_t - \frac{\beta\psi_t}{2 \sqrt{\delta_f}} & -\frac{\gamma\psi_t}{\sqrt{\delta_f}} \vspace{2mm}\\ \frac{\alpha\psi_t}{\sqrt{\delta_f}} & \phi_t + \frac{\beta\psi_t}{2\sqrt{\delta_f}}\end{pmatrix},\]
where $t\in\bR$ and $(\phi_t,\psi_t) \in\{(\cosh t,\sinh t),(\sinh t,\cosh t)\}$. We then see that
\[ G_f(\bZ) = \{\pm I_{2\times 2}\}\sqcup\{T\in O_f(\bZ) : T\mbox{ has infinite order}\}.\]
Hence, the first claim follows from Lemma~\ref{Dickson ambiguous} and the fact that $ax^2 + bxy + ay^2$ is $\GL_2(\bZ)$-equivalent to the form
\begin{equation}\label{ambiguous transformation} (2a-b)x^2 + (2a-b)xy + ay^2\mbox{ via } \begin{pmatrix}-1&-1\\1 & 0\end{pmatrix}.\end{equation}
Now, again by (\ref{O relation}), elements in $O_f(\bR)$ of finite order have the shape
\begin{equation}\label{indefinite finite order}\begin{pmatrix} \frac{-\beta}{\sqrt{D_f}} &  - \frac{2\gamma}{\sqrt{D_f}} \vspace{2mm}\\ \frac{2\alpha}{\sqrt{D_f}} & \frac{\beta}{\sqrt{D_f}}\end{pmatrix} \AND
\begin{pmatrix} \phi_t + \frac{\beta \psi_t}{2 \sqrt{\delta_f}} & \frac{\beta}{\alpha} \left(\phi_t + \frac{\beta\psi_t}{2 \sqrt{\delta_f}} \right) - \frac{\gamma\psi_t}{\sqrt{\delta_f}} \vspace{2mm}\\ - \frac{\alpha\psi_t}{\sqrt{\delta_f}} & - \phi_t - \frac{\beta\psi_t}{2 \sqrt{\delta_f}}, \end{pmatrix}\end{equation}
where $t\in\bR$ and $(\phi_t,\psi_t) \in\{(\cosh t,\sinh t),(\sinh t,\cosh t)\}$. Notice that the matrix on the left cannot lie in $\GL_2(\bZ)$ because $D_f$ is not square when $f$ is irreducible. Using the description of $O_{x^2-y^2}(\bR)$, it is then not hard to check that $[O_f(\bZ):G_f(\bZ)]\leq 2$, from which the second claim follows.
\end{proof}

\subsection{Proof of Theorem~\ref{negative Pell}} Suppose that $f(x,y) = \alpha x^2 + \beta xy - \alpha y^2$ and that $D_f$ is not a square. In the notation of Proposition~\ref{Of indefinite}, we have
\[\mbox{$x^2 - D_fy^2 = -4$ has integer solutions} \mbox{ if and only if }\det(T_{D_f}) = -1\]
by definition. But Proposition~\ref{Of indefinite} also implies that $\det(T_{D_f}) = -1$ is equivalent to
\[O_f(\bZ)\mbox{ has an element of finite order and negative determinant}.\]
The theorem now follows from Lemma~\ref{Dickson ambiguous} and (\ref{ambiguous transformation}). 

\subsection{Proof of Proposition~\ref{issue2 prop}} We shall need the following lemma.

\begin{lemma}\label{key corollary}For all $F\in V_{\bZ,f}^0$ with $\Delta(F)\neq\square$ and $T\in \GL_2(\bZ)\setminus\{\pm I_{2\times 2}\}$, we have
\begin{enumerate}[(a)]
\item $F_T \in V_{\bZ,f}^0$ if and only if $T\in O_f(\bZ)$,
\item $F_T = F$ if and only if $T = \pm D_f^{-1/2}M_f$.
\end{enumerate}
\end{lemma}
\begin{proof}Note that $F_T\in V_{\bZ,f_T}^0$ by (\ref{V bijection}). By Theorem~\ref{small char thm} (a), we then have $F_T\in V_{\bZ,f}^0$ if and only if $f_T = \pm f$, whence part (a) holds. By Theorem~\ref{small char thm} (a) and Proposition~\ref{auto theorem}, we  have $F_T = F$ if and only if $T$ is proportional to $M_f$, from which part (b) follows since $\det(T)=\pm1$.
\end{proof}

\subsubsection{The case when $f$ is positive definite or reducible} Let us first observe that:

\begin{lemma}\label{V in S lemma}We have $V_{\bZ,f}^0(X)\subset\S_f(X)$.
\end{lemma}
\begin{proof}Let $F\in V_{\bZ,f}^0(X)$ be given. If $f$ is positive definite, then clearly $F\in \S_f(X)$ by Lemma~\ref{Sf positive definite}. If $f$ is reducible, then recall Lemma~\ref{Sf reducible}, and we have $F\in \S_f(X)$ since
\[\frac{8C_F}{\beta^2}\in\bZ\AND \left\lvert\frac{8C_F}{\beta^2}\right\rvert\leq \left\lvert\frac{4(L_f(F)^2 + 4K_f(F))}{9}\right\rvert \leq \frac{20X}{9}\]
by (\ref{reducible generic}) and Proposition~\ref{explicit LK} (b), respectively.
\end{proof}

Lemma~\ref{V in S lemma} implies that part (a) holds. Together with Lemma~\ref{key corollary} (a, it further implies that for $F\in V_{\bZ,f}^0(X)$ with $\Delta(F)\neq\square$, the number of forms in $\S_f(X)$ which are $\GL_2(\bZ)$-equivalent to $F$ is equal to 
\[[O_f(\bZ) : \mbox{Stab}_{O_f(\bZ)}(F)].\]
By Lemma~\ref{key corollary} (b), we in turn have
\[\textstyle[O_f(\bZ) : \mbox{Stab}_{O_f(\bZ)}(F)] = [O_f(\bZ) : O_f(\bZ)\cap\{\pm I_{2\times 2},\pm  D_f^{-1/2}M_f\}],\]
which may be verified to be equal to $r_f$ using Propositions~\ref{Of pos def} and~\ref{Of reducible}.

\subsubsection{The case when $f$ is indefinite and irreducible} 

We shall use the notation from Lemma~\ref{Psi lemma}, Proposition~\ref{Of indefinite}, (\ref{T def}), and (\ref{J def}). Then, by definition, we have
\[ T_{D_f} = T_f^{-1}J_{k(f)} T^-(t_{D_f})T_f,\mbox{ where }k(f) = \begin{cases}1 &\mbox{if }u_{D_f}^2 - D_fv_{D_f}^2 = -4,\\
2&\mbox{if }u_{D_f}^2 - D_fv_{D_f}^2 = 4,\end{cases}\]
Now, by (\ref{para}) and (\ref{Phi- def}), a form in $V_{\bZ,f}^0(X)$ is of the shape
\[F = (F_{(L,K)}^{(i)})_{T^-(t)T_f},\mbox{ where }(L,K,t)\in\Omega^0(X)\times\bR\AND i\in\{1,2,3,4\}.\]
Observe that $J_1$ and $J_2$ commute with $T^-(t)$ as well as fix the forms in $V_{\bR,x^2-y^2}$. For any $n\in\bZ$, we then deduce that
\[ F_{T_{D_f}^n} =  (F_{(L,K)}^{(i)})_{T^-(t)J_{k(f)}^nT^-(nt_{D_f})T_f} =  (F_{(L,K)}^{(i)})_{T^-(t+nt_{D_f})T_f}.\]
Let $n_1\in\bZ$ be the unique integer such that $0\leq t+n_1t_{D_f} < t_{D_f}$. The existence of $n_1$ then implies part (a).\\

Next, suppose that $\Delta(F)\neq\square$, in which case
\[ \mbox{for $T\in\GL_2(\bZ)$}: F_T\in V_{\bZ,f}^0\mbox{ if and only if }T\in O_f(\bZ)\]
by Lemma~\ref{key corollary} (a). If $O_f(\bZ) = G_f(\bZ)$, then part (b) holds by the uniqueness of $n_1$. If $O_f(\bZ) \neq G_f(\bZ)$, then recall from Proposition~\ref{Of indefinite} that 
\[O_f(\bZ) = G_f(\bZ) \sqcup G_f(\bZ)M,\mbox{ where $M$ has finite order}.\]
From (\ref{O relation}), we see that
\[ M = \pm T_{f}^{-1}J_{k_0}T^-(t_0)T_f,\mbox{ where $t_0\in\bR$ and $k_0\in\{3,4\}$}.\]
Then,  for any $n\in\bZ$, it is straightforward to verify that 
\begin{align*} F_{T_{D_f}^nM} & = (F_{(L,K)}^{(i)})_{T^-(t+nt_{D_f})J_{k_0} T^-(t_0)T_f} \\&= 
\begin{cases}
(F_{(L,K)}^{(i)})_{T^-(-(t+nt_{D_f})+t_0)T_f} & \mbox{for }i\in\{1,2\},\\
(F_{(L,K)}^{(j)})_{T^-(-(t+nt_{D_f})+t_0)T_f} & \mbox{for $i\in\{3,4\}$, where $j\in\{3,4\}\setminus\{i\}$}.
\end{cases}\end{align*}
There is a unique $n_2\in\bZ$ such that $0\leq -(t+n_2t_{D_f}) + t_0< t_{D_f}$. Observe that
\[ F_{T_{D_f}^{n_1}} = F_{T_{D_f}^{n_2}M} \mbox{ would imply }F_{T_{D_f}^{n_1}}= (F_{T_{D_f}^{n_1}})_{T_{D_f}^{n_2-n_1}M}.\]
But $T_{D_f}^{n_2-n_1}M$ has finite order, and so it cannot proportional to $M_f$ by (\ref{indefinite finite order}), which is a contradiction by Lemma~\ref{key corollary} (b). Then, we conclude from Proposition~\ref{Of indefinite} that part (b) indeed holds.

\section{Error estimates and the main theorem}

Throughout this section, let $f(x,y) = \alpha x^2 + \beta xy + \gamma y^2$ be an integral and primitive binary quadratic form with $\Delta(f)\neq0$ and $\alpha>0$, in the shape (\ref{reducible f shape}) whenever $f$ is reducible. Let $D_f,r_f$ and $s_f$ be as in Theorem~\ref{Small Gal MT}.\\

In Subsections~\ref{error sec1} and~\ref{error sec2}, respectively, we shall first prove:

\begin{proposition}\label{error prop1}For any $\epsilon>0$, we have
\[\#\{F\in \S_f(X)\cap V_{\bZ,f}^0:L_f(F)^2 + 4K_f(F)=\square\} = O_{f,\epsilon}(X^{1+\epsilon}),\]
and 
\begin{align*} 
&\#\{F\in \S_f(X)\cap V_{\bZ,f}^0:(L_f(F)^2 + 4K_f(F))(2L_f(F) - K_f(F))/\Delta(f) =\square\\ &\hspace{8.5cm}\AND L_f(F)\neq0\} = O_f(X^{1/2+\epsilon}).
\end{align*}
Further, the number 
\[\#\{F \in \S_f(X) \cap V_{\bZ,f}^0 : - 4K_f(F)/\Delta(f) = \square \text { and } L_f(F) = 0\}\]
is equal to zero if $-\Delta(f) \ne \square$. and is bounded by $O_f(X)$ otherwise. 
\end{proposition}

Propositions \ref{error prop1}, \ref{abelian Gal prop}, and \ref{issue2 prop} then imply part d) of Theorem \ref{Small Gal MT}. \\

The reader should compare the last claim above with \cite[Theorem 1.4]{X2}. 

\begin{proposition}\label{error prop2} We have
\[\#\{F\in \S_f(X)\cap V_{\bZ,f}^0 : F\mbox{ is reducible}\} = \begin{cases} O_f(X(\log X)^2) & \mbox{if $f$ is irreducible},\\
O_f(X(\log X)^3)&\mbox{if $f$ is reducible}.\end{cases}\]
\end{proposition}

Now, from Propositions~\ref{issue2 prop},~\ref{error prop1}, and~\ref{error prop2}, we also easily see that
\begin{equation}\label{N1}N_{\bZ,f}^{(D_4)}(X) = \frac{1}{r_f}\#(S_f(X)\cap V_{\bZ,f}^0) + O_{f,\epsilon}(X^{1+\epsilon}) \mbox{ for any }\epsilon >0.\end{equation}
Let $\L_{f,w(f)}$ be a linear transformation on $\bR^3$ which takes $\Lambda_{f,w(f)}$ to $\bZ^3$, and define
\[ \R_f(X) = (\L_{f,w(f)}\circ\Theta_{w(f)})(\S_f(X)), \mbox{ where }w(f) = \begin{cases}
1 & \mbox{if $f$ is irreducible},\\ 2 & \mbox{if $f$ is reducible},
\end{cases}\]
as before. Observe that then
\[\#(\S_f(X)\cap V_{\bZ,f}^0) = \#(\Theta_{w(f)}(\S_f(X))\cap \Lambda_{f,w(f)}) = \#(\R_f(X)\cap \bZ^3).\]
By Proposition~\ref{issue1 prop}, we may apply Proposition~\ref{Davenport} to obtain
\begin{align}\label{N2}
\#(S_f(X)\cap V_{\bZ,f}^0) &= \Vol(\R_f(X)) + O(\max\{\Vol(\overline{\R_f(X)}),1\}) \\\notag
&= \frac{1}{\det(\Lambda_{f,w(f)})}\Vol(\Theta_{w(f)}(\S_f(X))) \\\notag
& \hspace{3cm}+ O_f(\max\{\Vol(\overline{\Theta_{w(f)}(\S_f(X))},1\}),
\end{align}
where by Proposition~\ref{det prop}, we know that
\[ \det(\Lambda_{f,w(f)}) = \begin{cases} s_f\alpha^3 &\mbox{if $f$ is irreducible},\\
s_f\beta^3/8&\mbox{if $f$ is reducible}.
\end{cases}\]
Hence, it remains to compute the above volumes, which we shall do in Subsection~\ref{proof sec}.

\subsection{Proof of Proposition~\ref{error prop1}}\label{error sec1}

Recall the notation from Proposition~\ref{explicit LK}. By definition and Proposition~\ref{LK integers}, we then have a well-defined map
\[ \iota:V_{\bZ,f}^0 \longrightarrow \bZ^3;\hspace{1em} \iota(F)= (L_f(F),L_{f,1}(F),L_{f,2}(F)).\]
Using Proposition~\ref{explicit LK}, it is easy to verify that $\iota$ is in fact injective. We shall also need the following result due to Heath-Brown \cite{HB1}.

\begin{lemma}\label{HB lemma}Let $\xi(x_1,x_2,x_3)$ be a ternary quadratic form such that its corresponding matrix $M_\xi$ has non-zero determinant. For $B_1,B_2,B_3>0$, let $N_\xi(B_1,B_2,B_3)$ denote the number of tuples $(x_1,x_2,x_3)\in\bZ^3$ such that
\[ |x_1|\leq B_1,\, |x_2|\leq B_2,\, |x_3|\leq B_3,\, \gcd(x_1,x_2,x_3)=1,\, \xi(x_1,x_2,x_3)=0.\]
Then, we have
\[ N_\xi(B_1,B_2,B_3) \ll_{\epsilon} \left(1+\left(B_1B_2B_3\cdot \frac{\det_0(M_\xi)^2}{|\det(M_\xi)|} \right)^{1/3+\epsilon}\right)d_3(|\det(M_\xi)|),\]
where $\det_0(M_\xi)$ denotes the greatest common divisor of the $2\times2$ minors of $M_\xi$, and $d_3(|\det(M_\xi)|)$ is the number of ways to write $|\det(M_\xi)|$ as a product of three positive integers.
\end{lemma}
\begin{proof}See \cite[Corollary 2]{HB1}.\end{proof}

In what follows, consider $F\in\S_f(X)\cap V_{\bZ,f}^0$, and for brevity, write
\[ (L,K,L_1,L_2) = (L_f(F),K_f(F),L_{f,1}(F),L_{f,2}(F)).\]
Since $\iota$ is injective, it is enough to estimate the number of choices for $(L,L_1,L_2)$. To that end, let us put $\D_f = \Delta(f)$. Recall from Propositions~\ref{explicit LK} and~\ref{LK integers} that
\[ L,K,L_1,L_2\in\bZ,\mbox{ as well as }L_1^2 - \D_fL_2^2 = 4\alpha^4(L^2 + 4K)/9,\]
which is non-zero by (\ref{Delta LK}). By the definition of our height, we also have
\begin{equation}\label{bounds} \begin{cases}L = O_f(X^{1/2})\AND K = O_f(X) &\mbox{in all cases},\\ L_1 = O_f(X^{1/2})\AND L_2 = O_f(X^{1/2})&\mbox{if $f$ is irreducible}.\end{cases}\end{equation}
The latter estimate holds by
\[\begin{cases}
(\ref{pos def para}), (\ref{general parameter}) &\mbox{if $f$ is positive definite},\\(\ref{indef para 1}), (\ref{indef para 2}), (\ref{general parameter}),\mbox{ and $0\leq t < t_{D_f}$} &\mbox{if $f$ is indefinite and irreducible},
\end{cases}\]
as well as the fact that $L_1$ and $L_2$ are linear in the coefficients of $F$. Finally, we shall write $d(-)$ for the divisor function. 

\begin{proof}[Proof of Proposition~\ref{error prop1}: first claim] Suppose that $L^2 + 4K=\square$. Then, we have
\[ L_1^2 - \D_f L_2^2 = U^2,\mbox{ where }U\in\bN\mbox{ is such that }U = O_f(X^{1/2}).\]
If $f$ is reducible, then $\D_f=\square$ and so clearly there are
\[ O_f\left(\sum_{U=1}^{X^{1/2}}d(U^2)\right) = O_{f,\epsilon}\left(\sum_{U=1}^{X^{1/2}}X^\epsilon\right) = O_{f,\epsilon}(X^{1/2+\epsilon})\]
choices for the pair $(L_1,L_2)$. If $f$ is irreducible, then note that
\[ (L_1/n)^2 - \D_f(L_2/n)^2 = (U/n)^2,\mbox{ where } n = \gcd(L_1,L_2,U),\]
and applying Lemma~\ref{HB lemma} to the ternary quadratic form $\xi$ with matrix 
\[ M_\xi = \begin{pmatrix}1 & 0 & 0 \\ 0 & -\D_f& 0 \\ 0 & 0 & -1\end{pmatrix},\mbox{ with }\begin{cases}\det(M_\xi) = \D_f,\\\det_0(M_\xi) = 1,\end{cases}\]
we deduce from (\ref{bounds}) that there are
\[ 
O_f\left(\sum_{n=1}^{X^{1/2}}N_\xi\left(\frac{X^{1/2}}{n}, \frac{X^{1/2}}{n},\frac{X^{1/2}}{n}\right)\right) = O_{f,\epsilon}\left(\sum_{n=1}^{X^{1/2}} \left(1+\frac{X^{1/2+\epsilon}}{n^{1+\epsilon}}\right)\right)= O_{f,\epsilon}(X^{1/2 + \epsilon}) \]
choices for the pair $(L_1,L_2)$. In both cases, we see that there are
\[ O_f(X^{1/2})\cdot O_{f,\epsilon}(X^{1/2+\epsilon}) = O_{f,\epsilon}(X^{1+\epsilon})\]
choices for $(L,L_1,L_2)$ in total, whence the claim.
\end{proof}

\begin{proof}[Proof of Proposition~\ref{error prop1}: second claim] Suppose that $(L^2+4K)(2L^2-K)/\D_f =\square$. By Proposition~\ref{LK integers}, we may write
\[ \gcd(L^2 + 4K, 4(2L^2 - K)/\D_f) = 9ma^2,\mbox{ where }m,a\in\bN\AND m\mbox{ is square-free}.\]
From the hypothesis, we then easily see that
\[ L^2 + 4K = 9mU^2\AND 4(2L^2-K)/\D_f =9mV^2,\mbox{ where }U,V \in \bN,\]
as well as that $m$ divides $L$. In particular, a simple calculation yields
\[ L^2 = m(U^2 + \D_fV^2),\mbox{ whence } mW^2 = U^2 + \D_fV^2,\mbox{ where } W\in\bZ\mbox{ with }L = mW.\]
Now, suppose also that $L\neq0$, in which case $m = O_f(X^{1/2})$ by (\ref{bounds}). Note also that
\[ m(W/n)^2 = (U/n)^2 + \D_f(V/n)^2,\mbox{ where }n = \gcd(W,U,V).\]
Applying Lemma~\ref{HB lemma} to the ternary quadratic form $\xi_m$ with matrix 
\[M_{\xi_m} = \begin{pmatrix}m & 0 & 0 \\ 0 & -1 & 0 \\ 0 & 0 & -\D_f\end{pmatrix},\mbox{ with }\begin{cases}\det(M_{\xi_m}) = m\D_f,\\\det_0(M_{\xi_m}) = \gcd(m,\D_f) \leq |\D_f|,\end{cases}\]
we then see from (\ref{bounds}) that there are
\begin{align*} O_f\left(\sum_{n=1}^{X^{1/2}/m}N_{\xi_m}\left(\frac{X^{1/2}}{mn}, \frac{X^{1/2}}{m^{1/2}n},\frac{X^{1/2}}{m^{1/2}n}\right)\right) &= O_{f,\epsilon}\left(\sum_{n=1}^{X^{1/2}/m}\left(1+\frac{X^{1/2+\epsilon}}{(mn)^{1+\epsilon}}\right)m^\epsilon\right)\\
& = O_{f,\epsilon}\left( \frac{X^{1/2}}{m^{1-\epsilon}} + \frac{X^{1/2+\epsilon}}{m}\right)\end{align*}
choices for $(x,u,v)$ when $m$ is fixed. It follows that we have
\[O_{f,\epsilon}\left(\sum_{m=1}^{X^{1/2}}\left(\frac{X^{1/2}}{m^{1-\epsilon}} + \frac{X^{1/2+\epsilon}}{m}\right)\right)=O_{f,\epsilon}\left(X^{1/2+\epsilon}\right)\]
choices for $(m,x,u,v)$ and hence for $(L,K)$.\\

Next, regard $(L,K)$ as being fixed, and recall that 
\[ L_1^2 - \D_fL_2^2 = T, \mbox{ where }T = 4\alpha^4(L^2+4K)/9.\]
We claim that there are $O_f(d(T))$ choices for $(L_1,L_2)$. If $f$ is positive definite or if $f$ is reducible, then this is clear. If $f$ is indefinite and irreducible, then by Definition~\ref{S(X) def} as well as Propositions~\ref{LK invariant} and \ref{Phi-}, we have
\[ F = (\Psi_f\circ\Phi^{(i)})(L,K,t),\mbox{ where }0\leq t<t_{D_f}\AND i\in\{1,2,3,4\}.\]
Since $\D_f>0$, we must have $L^2 + 4K>0$ by the hypothesis, and so in fact $i\in\{1,2\}$. From the proof of Lemma~\ref{Sf indefinite}, we know that
\[ L_1 - \sqrt{D_f}L_2  = (-1)^i\sqrt{T}e^{4t}\AND L_1 + \sqrt{D_f}L_2  = (-1)^i\sqrt{T}e^{-4t},\]
which implies that 
\[ L_1 = (-1)^i\sqrt{T}\cosh(4t)\AND L_2 = (-1)^i\sqrt{T}\sinh(4t)/\sqrt{D_f}.\]
Since $t = O_f(1)$, we then deduce that indeed there are $O_f(d(T))$ choices for $(L_1,L_2)$. Using the bound $d(T) = O_{\epsilon}(T^\epsilon) = O_{f,\epsilon}(X^{\epsilon})$, we conclude that there are
\[ O_{f,\epsilon}(X^{1/2+\epsilon})\cdot O_{f,\epsilon}(X^{\epsilon}) = O_{f,\epsilon}(X^{1/2+\epsilon})\]
choices for $(L,L_1,L_2)$ in total, whence the claim. 
\end{proof}

\begin{proof}[Proof of Proposition~\ref{error prop1}: third claim] Suppose that $L=0$ and that $F$ is in the shape as in (\ref{abc family}). Using Proposition~\ref{explicit LK}, we then deduce that
\[C = (-12 \gamma A + 3 \beta B)/(2 \alpha),\mbox{ and so }K = -9\D_f(\alpha B^2 - 4\beta AB + 16\gamma A^2)/(4\alpha^3).\]
Hence 
\[(L^2 + 4K)(2L^2 - K)/81\D_f = -4 \D_f (\alpha B^2 - 4 \beta AB + 16 \gamma A^2)^2/(4 \alpha^3)^2,\]
from which it follows that the above expression is a square if and only if $-\D_f$ is a square. This also follows immediately from the observation that the above product is equal to $-4K^2/\D_f$ in this case. \\

We now suppose that $-\Delta(f) = \square$, so in particular $f$ is positive definite. $F$ is then determined by $(A,B)\in\bZ^2$, and that $|K|\leq X$ implies
\[\left\lvert \left(B-\frac{2\beta}{\alpha}A\right)^2 - \frac{4\D_f}{\alpha^2}A^2\right\rvert\ll_f X.\]
Hence there are $O_f(X)$ choices for $(A,B)$. It follows that the claim holds.
\end{proof}

\subsection{Proof of Proposition~\ref{error prop2}}\label{error sec2}

By Lemma~\ref{type 1} and Proposition~\ref{error prop1}, we have
\begin{equation}\label{type1 estimate} \#\{F\in \S_f(X)\cap V_{\bZ,f}^0: F\mbox{ is reducible of type $1$}\} = O_{f,\epsilon}(X^{1+\epsilon}),\end{equation}
whence it is enough to consider the reducible forms in $\S_f(X)\cap V_{\bZ,f}^0$ of type 2; recall Definition~\ref{reducible types def}. By definition, such a form has the shape
\[ F(x,y) =  p_2q_2 x^4 + (p_2q_1 + p_1q_2)x^3y + (p_2q_0 + p_1q_1 + p_0q_2)x^2y^2 + (*)xy^3 + (*)y^4,\]
where $p_2,p_1,p_0,q_2,q_1,q_0\in\bZ$, and we have
\[p_0 = (\beta p_1 - 2\gamma p_2)/(2\alpha)\AND q_0 = (\beta q_1 - 2\gamma q_2)/(2\alpha)\]
by Lemma~\ref{quadratic factor}. We have the condition
\begin{align}\label{pq condition1}
&|(\alpha p_1^2 - 2\beta p_1p_2 + 4\gamma p_2^2)/\alpha|,|(\alpha q_1^2 - 2 \beta q_1 q_2 + 4 \gamma q_2^2)/\alpha|, \\\notag
&\hspace{4.5cm}|p_2|,|\alpha p_1 - \beta p_2|,|q_2|,\lvert \alpha q_1 - \beta q_2|\geq1
\end{align}
since the above numbers are all integers. Using Proposition~\ref{explicit LK} (a), we compute that
\[\dfrac{L_f(F)^2 + 4K_f(F)}{9} = \dfrac{\alpha p_1^2 - 2 \beta p_1 p_2 + 4 \gamma p_2^2}{\alpha}\cdot\dfrac{\alpha q_1^2 - 2 \beta q_1 q_2 + 4 \gamma q_2^2}{\alpha}.\]
Now, by the definition of our height, we clearly have
\begin{equation}\label{pq condition2}
|(\alpha p_1^2 - 2\beta p_1p_2 + 4\gamma p_2^2)/\alpha|,|(\alpha q_1^2 - 2 \beta q_1 q_2 + 4 \gamma q_2^2)/\alpha|\leq X.
\end{equation}
Observe also that
\begin{equation}\label{pq condition3}p_2q_2,p_2q_1 + p_1q_2,p_1q_1 = O_f(X^{1/2})\mbox{ if $f$ is indefinite and irreducible}\end{equation}
by (\ref{indef para 1}), (\ref{indef para 2}), (\ref{general parameter}), and the bound $0\leq t < t_{D_f}$. We then deduce that 
\begin{equation}\label{reducible type 2 bound}
\#\{F\in \S_f(X)\cap V_{\bZ,f}^0: F\mbox{ is reducible of type $2$}\}
\leq \#(\R_f'(X)\cap \bZ^4),
\end{equation}
where we define
\[ \R_f'(X) = \{(p_2,p_1,q_2,q_1)\in\bR^4: (\ref{pq condition1}),\,(\ref{pq condition2}),\AND (\ref{pq condition3})\}.\]
It is clear that this set is bounded and semi-algebraic. Hence, we may apply Proposition~\ref{Davenport} to estimate the number of integral points it contains.

\subsubsection{The case when $f$ is irreducible} Let us define
\[\R_f''(X) =\L_{D_f}(\R'_f(X)),\mbox{ where }\L_{D_f} =  \left(\begin{smallmatrix}\sqrt{D_f}&0&0&0\\-\beta & \alpha & 0 & 0 \\ 0 & 0& \sqrt{D_f} & 0\\ 0 & 0 & -\beta & \alpha \end{smallmatrix}\right).\]
Applying Proposition~\ref{Davenport}, we then obtain
\begin{align*}
\#(\R_f'(X)\cap\bZ^4) & = \Vol(\R_f'(X)) + O(\max\{\Vol(\overline{\R_f(X)},1\})\\
& = \frac{1}{\det(\L_{D_f})}\Vol(\R_f''(X)) + O_f(\max\{\Vol(\overline{\R_f''(X)}),1\})
\end{align*}
For any $(u_2,u_1,v_2,v_1) \in \R_f''(X)$, from (\ref{pq condition1}) and (\ref{pq condition2}), we deduce that
\[|u_2|,|u_1|,|v_2|,|v_1|\geq 1\]
as well as that
\begin{equation}\label{uv condition}
\begin{cases} 
1\leq|u_1^2+ u_2^2|,|v_1^2+v_2^2|\leq \alpha^4X&\mbox{if $f$ is positive definite},\\
1\leq|u_1^2 - u_2^2|,|v_1^2-v_2^2|\leq \alpha^4X&\mbox{if $f$ is indefinite}.
\end{cases}
\end{equation}
This, together with (\ref{pq condition3}), implies that in fact
\[ 1\leq |u_2|,|u_1|,|v_2|,|v_1|, |u_2v_2|,|u_1v_1|\ll_f X^{1/2}.\]
We then compute that
\begin{align*}\Vol(\R_f''(X))  &= O_f\left( \prod_{i=1}^{2}\int_{1}^{X^{1/2}/v_i} du_i dv_i\right)= O_f(X(\log X)^2),\\
 \Vol(\overline{\R_f''(X)}) &= O_f(X\log X).\end{align*}
The claim now follows from (\ref{type1 estimate}) and (\ref{reducible type 2 bound}).

\subsubsection{The case when $f$ is reducible}

Let us define
\[ \R''_f(X) = \L_{0,D_f}(\R'_f(X)),\mbox{ where }\L_{0,D_f} =\left(\begin{smallmatrix}1 &1&0&0\\[0.75ex]-1&1&0&0 \\[0.75ex] 0&0&1&1\\[0.75ex] 0&0&-1&1 \end{smallmatrix}\right) \left(\begin{smallmatrix}\sqrt{D_f}&0&0&0\\-\beta & \alpha & 0 & 0 \\ 0 & 0& \sqrt{D_f} & 0\\ 0 & 0 & -\beta & \alpha \end{smallmatrix}\right).\]
Since $D_f=\square$ in this case, we see that
\[ \L_{0,D_f}(\R_f'(X)\cap\bZ^4)\subset \R_f''(X)\cap\bZ^4\mbox{ and so }\#(\R_f'(X)\cap \bZ^4) \leq \#(\R_f''(X)\cap\bZ^4).\]
Now, applying Proposition~\ref{Davenport}, we have
\[\#(\R_f''(X)\cap\bZ^4) = \Vol(\R_f''(X)) + O(\max\{\Vol(\overline{\R_f''(X)}),1\}).\]
For any $(z_1,z_2,z_3,z_4)\in\R_f''(X)$, the conditions (\ref{pq condition1}) and (\ref{pq condition2}) imply that
\[|z_1|,|z_2|,|z_3|,|z_4|\geq 1\AND |z_1z_2z_3z_4|\leq \alpha^4X,\]
which is analogous to (\ref{uv condition}). We then compute that
\begin{align*}\Vol(\R_f''(X))  &= O_f\left(\int_1^{X} \int_1^{\frac{X}{z_4}} \int_{1}^{\frac{X}{z_3 z_4}} \int_1^{\frac{X}{z_2 z_3 z_4}} d z_1 dz_2 dz_3 dz_4\right) = O_f(X(\log X)^3),\\
\Vol(\overline{\R_f''(X)}) &= O_f(X(\log X)^2).\end{align*}
The claim now follows from (\ref{type1 estimate}) and (\ref{reducible type 2 bound}).

\subsection{Proof of Theorem~\ref{Small Gal MT}}\label{proof sec}

We have already proven part (d). To prove parts (a) through (c), it remains to compute the volumes in (\ref{N2}).

\subsubsection{The case when $f$ is positive definite} We have
\[\Vol(\Theta_{1}(\S_f(X))) = \frac{8\alpha^3}{D_f^{3/2}}\cdot\frac{1}{18}\cdot\Vol(\Omega^+(X)\times[-\pi/4,\pi/4))\]
by Lemma~\ref{Psi lemma} and Proposition~\ref{Phi+} (b), as well as
\[\Vol (\Omega^+(X)\times[-\pi/4,\pi/4)) = \int_{-X^{1/2}}^{X^{1/2}} \int_{-L^2/4}^X  \frac{\pi}{2} dK dL
=\frac{13\pi}{12}X^{3/2}.\]
Observe also that
\[\Vol(\overline{\Theta_{1}(\S_f(X))}) = O_f(X)\]
because $\Theta_{1}(\S_f(X))$ lies in the cube centered at the origin of side length $O_f(X^{1/2})$ by (\ref{pos def para}) and (\ref{general parameter}). We then deduce part (a) from (\ref{N1}) and (\ref{N2}).

\subsubsection{The case when $f$ is reducible} We have
\[ \Vol(\Theta_2(\S_f(X))) =\frac{1}{18}\cdot 2\cdot \Vol(\Omega^0(X)\times [t_{f,1},t_{f,2}])\]
by Proposition~\ref{Phi f}, as well as 
\[\Vol(\Omega^0(X)\times [t_{f,1},t_{f,2}])= \int_{-X^{1/2}}^{X^{1/2}}\int_{-X}^X \frac{1}{4}\log\left(\frac{20X}{9}\right)dKdL= X^{3/2}\log(20X/9).\]
We then deduce part (b) from Lemma~\ref{S bar} below as well as (\ref{N1}) and (\ref{N2}).

\begin{lemma}\label{S bar}We have $\Vol(\overline{\Theta_{2}(\S_f(X))}) = O_f(X^{3/2})$.
\end{lemma}
\begin{proof}By Definition~\ref{S(X) def}, an element in $\Theta_{2}(\S_f(X))$ takes the form
\[ (A,B,C) = (\Theta_2\circ\Phi_f)(L,K,t), \mbox{ where }(L,K,t)\in \Omega^0(X)\times [t_{f,1},t_{f,2}].\]
Let us recall that
\begin{equation}\label{red bounds}|L|\leq X^{1/2},\, |K|\leq X,\, 4t_{f,1} =-\log 8,\, 4t_{f,2} = \log(5X/18).\end{equation}
Then, from (\ref{red para}), we see that $1$-dimensional projections of $\Theta_{2}(\S_f(X))$ have lengths of order $O_f(X)$. As for the $2$-dimensional projections, note that (\ref{para}) and (\ref{red bounds}) yield
\[|C| = \beta^2 e^{4t}\AND 1\ll_f |C|\ll_f X,\]
as well as the estimates
\[\left\lvert B - \frac{6\alpha^2C}{\beta^2}\right\rvert \leq \frac{1}{2}X^{1/2}\AND \left\lvert A -\frac{\alpha^4C}{\beta^4} \right\rvert \leq\frac{5}{144|C|}X + \frac{\alpha^2}{2\beta^2}X^{1/2}.\]
Hence, the projections of $\Theta_2(\S_f(X))$ onto the $BC$-plane and $AC$-plane, respectively, have areas bounded by 
\[O_f\left(\int_1^{X} X^{1/2} dC\right) \AND
O_f\left(\int_1^{X}\left(\frac{1}{C}X+ X^{1/2}\right)dC\right).\]
Similarly, from (\ref{para}) and (\ref{red bounds}), we deduce that
\[|2B-L| = 12\alpha^2 e^{4t},\, 1\ll_f |2B-L|\ll_f X,\, |B|\ll_f X,\]
as well as the estimate
\[ \left\lvert A - \frac{\alpha^2B}{6\beta^2}\right\rvert
\leq \frac{5\alpha^2}{12\beta^2}\left(\frac{1}{|2B-L|} X+ X^{1/2}\right).\]
Note that $|L|\leq X^{1/2}$ also implies that
\[ |2B - L| \geq |2|B| - |L|| \geq 2|B| - X^{1/2} \mbox{ when }|B|\geq X^{1/2}/2.\]
Hence, the projection of $\Theta_2(\S_f(X))$ onto the $AB$-plane has area bounded by
\[ O_f\left(\int_{0}^{1+X^{1/2}/2}(X+X^{1/2})dB + \int_{1+X^{1/2}/2}^{X}\left(\frac{1}{2B-X^{1/2}}X+X^{1/2}\right)dB \right).\]
It follows that all of the $2$-dimensional projections of $\Theta_2(\S_f(X))$ have areas of order $O_f(X^{3/2})$, and this proves the lemma.\end{proof}

\subsubsection{The case when $f$ is indefinite and irreducible} We have
\[ \Vol(\Theta_{1}(\S_f(X))) = \frac{8\alpha^3}{D_f^{3/2}}\cdot\frac{1}{18}\cdot 2\cdot\left(\Vol(\Omega^+(X)\times [0,t_{D_f})) + \Vol(\Omega^-(X)\times [0,t_{D_f}))\right)\]
by Lemma~\ref{Psi lemma} and Proposition~\ref{Phi-}, as well as
\begin{align*} \Vol(\Omega^+(X)\times [0,t_{D_f})) &= \int_{-X^{1/2}}^{X^{1/2}} \int_{-L^2/4}^X  t_{D_f} dK dL
=\frac{13t_{D_f}}{6}X^{3/2},\\[0.5ex]
\Vol(\Omega^-(X)\times [0,t_{D_f})) &= \int_{-X^{1/2}}^{X^{1/2}} \int_{-X}^{-L^2/4} t_{D_f} dK dL = \frac{11 t_{D_f}}{6}X^{3/2},\end{align*}
Observe also that 
\[\Vol(\overline{\Theta_{1}(\S_f(X))}) = O_f(X)\]
because $\Theta_{1}(\S_f(X))$ lies in the cube centered at the origin of side length $O_f(X^{1/2})$ by (\ref{indef para 1}), (\ref{indef para 2}), (\ref{general parameter}), and the bound on $t$. We then deduce part (c) from (\ref{N1}) and (\ref{N2}).

\section{Acknowledgments}

The first-named author was partially supported by the China Postdoctoral Science Foundation Special Financial Grant (grant number: 2017T100060). We would like to thank the referee for many useful suggestions which helped improve the exposition of the paper significantly.

\end{document}